\documentclass[reqno,12pt]{amsart}
\usepackage{dsfont}
\usepackage{amssymb,mathtools,psfrag,graphicx}
\usepackage{epstopdf} 

\usepackage[colorlinks=true]{hyperref}
\hypersetup{urlcolor=blue, citecolor=red}

\usepackage[OT1]{fontenc}
\usepackage[latin1]{inputenc}

\newcommand{\be}{\begin{equation}}
\newcommand{\ee}{\end{equation}}
\newcommand{\bea}{\begin{eqnarray}}
\newcommand{\eea}{\end{eqnarray}}
\newcommand{\bu}{\mathbf u}
\newcommand{\bv}{\mathbf v}
\newcommand{\bw}{\mathbf w}

\newcommand{\f}{\mathbf f}
\newcommand{\bfe}{\mathbf e}
\newcommand{\fKU}{\overline{\Pi_{t_0}^{-1}K\cap\mathcal{U}}}
\newcommand{\fKUI}{\overline{\Pi_{t_0}^{-1}K\cap\mathcal{U}_I}}
\newcommand{\KU}{\Pi_{t_0}^{-1}K\cap\mathcal{U}}
\newcommand{\KUI}{\Pi_{t_0}^{-1}K\cap\mathcal{U}_I}

\newcommand{\mA}{\mathcal A}
\newcommand{\mB}{\mathcal B}
\newcommand{\mC}{\mathcal C}
\newcommand{\mD}{\mathcal D}

\newcommand{\mK}{\mathcal K}
\newcommand{\mM}{\mathcal M}

\newcommand{\mP}{\mathcal P}

\newcommand{\mU}{\mathcal U}
\newcommand{\mV}{\mathcal V}
\newcommand{\mW}{\mathcal W}
\newcommand{\mX}{\mathcal X}

\newcommand{\wconv}{\stackrel{wsc}{\rightharpoonup}}
\newcommand{\rd}{{\text{\rm d}}}
\newcommand{\Cloc}{\mathcal{C}_{\text{loc}}}
\DeclareMathOperator{\Exp}{e}
\theoremstyle{plain}
\newtheorem{thm}{Theorem}[section]
\newtheorem{lem}{Lemma}[section]
\newtheorem{prop}{Proposition}[section]

\newtheorem{defs}{Definition}[section]
\newtheorem{prob}{Problem}[section]
\theoremstyle{definition}
\newtheorem{rmk}{Remark}[section]

\newcommand{\comment}[1]{\\\indent\textcolor{red}{\framebox{\parbox{0.9\textwidth}{\textbf{#1}}}}\\}

\renewcommand{\comment}[1]{}

\begin{document}

\title[Trajectory statistical solutions for 3D-NSE-like systems]{Trajectory Statistical Solutions for three-dimensional Navier-Stokes-like systems}
\author{Anne C. Bronzi}
\author{Cecilia F. Mondaini}
\author{Ricardo M. S. Rosa}
\address{Instituto de Matem\'atica, Universidade Federal do Rio de Janeiro, \\
C.P. 68530, 21.941-909, Rio de Janeiro, RJ, Brazil}

\email[A. Bronzi]{annebronzi@gmail.com}
\email[C. Mondaini]{cfmondaini@gmail.com}
\email[R. Rosa]{rrosa@im.ufrj.br}

\date{July 23, 2013}

\thanks{This work was partly supported by CNPq-Brazil, under the grant ``Edital Universal'' 500437/2010-6. The three authors were also individually supported by scholarships from CNPq-Brazil.}


\begin{abstract}
A general framework for the theory of statistical solutions on trajectory spaces is constructed for a wide range of equations involving incompressible viscous flows. This framework is constructed with a general Hausdorff topological space as the phase space of the system, and with the corresponding set of trajectories belonging to the space of continuous paths in that phase space. A trajectory statistical solution is a Borel probability measure defined on the space of continuous paths and carried by a certain subset which is interpreted, in the applications, as the set of solutions of a given problem. The main hypotheses for the existence of a trajectory statistical solution concern the topology of that subset of ``solutions'', along with conditions that characterize those solutions within a certain larger subset (a condition related to the assumption of strong continuity at the origin for the Leray-Hopf weak solutions in the case of the Navier-Stokes and related equations). The aim here is to raise the current theory of statistical solutions to an abstract level that applies to other evolution equations with properties similar to those of the three-dimensional Navier-Stokes equations. The applicability of the theory is illustrated with the B\'enard problem of convection in fluids.
\end{abstract}


\maketitle

\section{Introduction}

The concept of statistical solutions has emerged in the context of fluid dynamics in order to provide a rigorous mathematical definition for the notion of ensemble average, commonly used in the study of turbulent flows. In such flows, the relevant physical quantities (e.g., velocity, kinetic energy, and pressure) present a wild variation in space and time, characterizing a highly irregular and unpredictable behavior. Nevertheless, those quantities display a regular behavior when considered with respect to some average. In an attempt to investigate general properties of such flows, one is then naturally led to deal with averages of the desired quantities. Several types of averages are usually considered, such as locally in space, locally in time, and with respect to an ensemble of experiments. Statistical solutions are directly related to this latter notion of average, known as ensemble average.

In the 1970's, two main definitions of statistical solutions for the Navier-Stokes equations have been developed. First, Foias \cite{Foias72} introduced the notion of statistical solutions in the phase space, consisting of a family of measures parametrized by the time variable representing the evolution of probability distributions of a viscous incompressible fluid. A few years later, Vishik and Fursikov \cite{VF77} introduced a different notion of statistical solutions, based on a single measure defined on the space of trajectories. More recently, Foias, Rosa and Temam \cite{FRT2010,FRT} (see also \cite{FMRT}) introduced a slightly modified definition of this latter solution, inspired by the definition given in \cite{VF77}, and which was denoted as Vishik-Fursikov measure, for a measure defined in the space of trajectories. Projecting this measure to the phase space at each time, they obtained a particular type of statistical solution, which is termed a Vishik-Fursikov statistical solution. What is interesting about this new definition is that every Vishik-Fursikov statistical solution is a statistical solution in the sense of Foias-Prodi. Besides being more favorable to analysis, the former seems to possess additional properties.

Since its initial steps, the theory of statistical solutions has gone through a significant development, becoming a subject that encompasses a number of concepts from several different areas of Mathematics and with a growing number of applications \cite{Foias74,CFM94,Bercovicietal95,Foias97,Fursikov99,FMRT01,RRT08, BronziRosa2014}. The idea in our work is to extend this theory to a more abstract level, so that similar results could be obtained for other equations that share the same potential pathologies of the Navier-Stokes equations (e.g., the lack of a uniqueness result and some peculiar properties of the Leray-Hopf weak solutions). We aimed at extracting the key ideas given in \cite{FRT2010,FRT} and adapt them to a framework as general as possible. In this sense we are very much indebted to the previous fundamental works \cite{Foias72, VF77} and, more recently, \cite{FRT2010,FRT}.

This is in fact a bold idea encompassing several difficult parts. The current work can be viewed as the first part in this larger project. Here, we focus only on the concept of Vishik-Fursikov measures defined in \cite{FRT2010,FRT}, calling them trajectory statistical solutions, and adapting them to an abstract level suitable to a wide range of applications. We are also working on the corresponding result for statistical solutions in phase space, generalizing the notion of Vishik-Fursikov statistical solution defined in \cite{FRT2010,FRT}. This result will be presented elsewhere. Applications of this general framework for the convergence of statistical solutions for models depending on a parameter are also under development.

We now give an outline of the present work: we start by defining a general notion of statistical solutions that incorporates the idea given in \cite{FRT2010,FRT} of a Vishik-Fursikov measure (Definition \ref{def-stat-sol}). This definition is built over an abstract framework based on a general Hausdorff topological space $X$ and the associated space of continuous paths $\mC(I,X)$, over a given time interval $I\subset \mathds{R}$ (i.e. the space of continuous functions defined on a real interval $I$ and assuming values in $X$), and endowed with the compact-open topology. A key object in this theory is a subspace $\mU$ of $\mC(I,X)$, which has no special meaning in this abstract level, but which, in the applications, is taken to be the set of (individual) solutions of a given evolutionary system, for which $X$ is a phase space. A general concept of statistical solution is then defined with respect to this set $\mU$ and called a $\mU$-trajectory statistical solution, which consists of a tight Borel probability measure on the space of continuous paths in $X$ and carried by $\mU$.

Our main concern is to prove the existence of a trajectory statistical solution for a certain initial value problem (Problem \ref{initvalueprob}). Although there is no equation at the abstract level, we consider an interval $I$ closed and bounded at the left and which is interpreted as a time interval. Then, a trajectory statistical solution is sought as a measure in $\mC(I,X)$ which is carried by $\mU$ and such that its projection at the left end point of $I$ is equal to a given ``initial'' Borel probability measure defined on $X$.

In order to obtain an existence theorem for this abstract initial-value problem, a series of restrictions must be imposed on the set $\mU$ that mimic some essential properties of the space of solutions occurring in the applications (Definition \ref{hypothesisH}). Under those hypotheses, we prove the existence of a $\mU$-trajectory statistical solution for the initial value problem for any initial tight Borel probability measure on $X$ (Theorem \ref{existence}).

The trajectory statistical solution of our initial value problem is obtained through the limit of an approximating net of measures, so that one of the tools we need is a compactness result for measures. A result of this kind which is suitable for our abstract context was developed by Topsoe \cite{Topsoe,Topsoebook,Topsoe2} in his works on a generalization of Prohorov's Theorem (see \cite{Prohorov56}). One of his results states that the uniform tightness of a family of tight Borel probability measures defined on a general Hausdorff space implies that the family is compact with respect to a certain topology which is stronger than the classical weak-star topology for measures. This stronger topology is based on semi-continuity, rather than continuity (see Section \ref{sectopmeasures}), and is not strictly necessary, but it is a more general and stronger result and it simplifies our presentation.

This abstract formulation applies to three-dimensional Navier-Stokes equations considered in the original works and extends to a large number of equations involving incompressible viscous fluids. We illustrate this fact by applying this theory to the well-known B\'enard problem of convection in fluids, whose determining equations consists of a particular coupling of the Navier-Stokes equations with an equation for the temperature.

\section{Basic Tools}

We present below the concepts and results that form the mathematical background of our work. In the first subsection, we present our basic functional space and define some useful operators. The second subsection is dedicated to the introduction of the measure spaces we work on and of the notion of \emph{tightness}. Finally, the third one is concerned with the definition of a topology for these measure spaces, which is needed in order to make sense of the convergence of nets. Most of the ideas in this latter subsection were introduced by Topsoe \cite{Topsoe,Topsoebook,Topsoe2}.

\subsection{Functional and topological preliminaries}

Throughout this work we consider a Hausdorff topological space $X$ and an arbitrary interval
$I\subset\mathbb{R}$. Denote by $\mX=\Cloc(I,X)$ the \emph{space of continuous paths} in $X$ endowed with the \emph{compact-open topology}, i.e. the space of continuous functions from $I$ into $X$ with the topology generated by the subbase consisting of the subsets
\begin{equation}
  \label{defSJU}
  S(J,U)=\{u\in\mX\,|\,u(J)\subset U\},
\end{equation}
where $J\subset I$ is a compact subinterval and $U\subset X$ is an open set. With respect to this topology, $\mX$ is a Hausdorff topological vector space.

The subscript ``\emph{loc}'' in $\Cloc(I,X)$ refers to the fact that this topology considers compact sets in $I$. When $X$ is a uniform space, the compact-open topology in $\mX$ coincides with the \emph{topology of uniform convergence on compact subsets} \cite[Theorem 7.11]{Kelley75}. This holds, in particular, when $X$ is a topological vector space, which is often the case in applications, such as the one that we present in Section \ref{secapplication}.

For any $t\in I$, let $\Pi_{t}:\mX\rightarrow X$ be the ``projection'' map at time $t$ defined by
\[\Pi_{t}u = u(t), \quad \forall u\in\mX.\]
It is not difficult to check that $\Pi_{t}$ is continuous with respect to the compact-open topology.

\subsection{Measures and tightness}

Let $\mB(X)$ be the $\sigma$-algebra of Borel sets in $X$. We denote by $\mathcal{M}(X)$ the set of finite Borel measures in $X$, i.e. the set of measures $\mu$ defined on $\mB(X)$ such that $\mu(X)<\infty$. The subset of $\mM(X)$ consisting of Borel probability measures is denoted by $\mathcal{P}(X)$.

A \emph{carrier} of a measure is any measurable subset of full measure, i.e. such that its complement has null measure. If $C$ is a carrier for a measure $\mu$, we say that $\mu$ is \emph{carried} by $C$. If the carrier is a single point $x\in X$, the probability measure is a \emph{Dirac measure} and is denoted $\delta_x$. A probability measure that can be written as a convex combination of Dirac measures is called a \emph{discrete measure}.

We say that a Borel measure $\mu$ on $X$ is \textit{tight} if, for every $A\in\mathcal{B}(X)$,
\[\mu(A)=\sup\{\mu(K)\,|\,K\text{ compact and }K\subset A\}.\]
The set of measures $\mu\in\mathcal{M}(X)$ (resp., $\mu\in\mathcal{P}(X)$) which are tight is denoted by $\mathcal{M}(X,t)$ (resp., $\mathcal{P}(X,t)$).

Furthermore, a net $(\mu_\alpha)_\alpha$ of measures in $\mathcal{M}(X)$ is said to be \textit{uniformly tight} if for every $\varepsilon>0$ there exists a compact set $K\subset X$ such that
\[\mu_\alpha(X\backslash K)<\varepsilon, \quad \forall \alpha.\]

Now consider a Hausdorff space $Y$ and let $F:X\rightarrow Y$ be a Borel measurable function. Then for every measure $\mu$ on $\mathcal{B}(X)$ we can define a measure $F\mu$ on $\mathcal{B}(Y)$ as
\[F\mu(E)=\mu(F^{-1}(E))\,,\,\,\forall E\in\mathcal{B}(Y),\]
which is called the \textit{induced measure from $\mu$ by $F$} on $\mB(Y)$. It turns out that when $\mu$ is a tight measure and $F$ is a continuous function, the induced measure $F\mu$ is also tight.

In regard to the concept of induced measures, we also mention the well-known result that if $\varphi:Y\rightarrow\mathbb{R}$ is a $F\mu$-integrable function then $\varphi\circ F$ is $\mu$-integrable and
\be\label{eq00}
\int_X\varphi\circ F\rd\mu=\int_Y\varphi \rd F\mu.
\ee

For the sake of notation, if $\mu\in\mM(X)$ and $f$ is a $\mu$-integrable function, we write
\[\mu(f)=\int_X f \rd\mu.\]

\subsection{A topology for the set of measures}
\label{sectopmeasures}

In \cite{Topsoebook}, Topsoe considered a topology in $\mM(X)$ obtained as the smallest one for which the mappings $\mu\mapsto \mu(f)$ are upper semicontinuous, for every bounded, real-valued and upper semicontinuous function $f$ on $X$. Topsoe calls this topology the ``weak topology'', but in order to avoid any confusion we call it here the \emph{weak semi-continuity topology} on $\mathcal{M}(X)$. If a net $(\mu_\alpha)_\alpha$ converges to $\mu$ with respect to this topology, we denote $\mu_\alpha\wconv\mu$.

A more common topology used in $\mM(X)$ is the \emph{weak-star topology,} which is the smallest topology for which the maps $\mu\mapsto \mu(f)$ are continuous, for every bounded, real-valued and continuous function $f$ on $X$. According to Lemma \ref{portmanteau} below, the weak-star topology is, in general, weaker than the weak semi-continuity topology defined above, and they both coincide when $X$ is a completely regular Hausdorff topological space.

Although our framework is based on a general Hausdorff space, the proofs rely on reducing some structures to compact subspaces, hence completely regular, in which case both topologies coincide, so that we could have very well considered only the weak-star topology. However, we prefer to use the weak semi-continuity topology since it is a more natural topology for arbitrary Hausdorff spaces which simplifies our presentation and yields a compactness result in a stronger topology.

The following result provides some useful characterizations for the weak semi-continuity topology (see \cite[Theorem 8.1]{Topsoebook}). Recall that a topological space $X$ is \emph{completely regular} if every nonempty closed set and every singleton disjoint from it can be separated by a continuous function.

\begin{lem}\label{portmanteau}Let $X$ be a Hausdorff space.
For a net $(\mu_\alpha)_\alpha$ in $\mathcal M(X)$ and $\mu\in \mathcal M(X)$,
consider the following statements:
\begin{itemize}
\item [(1)]$\mu_\alpha \wconv\mu$;
\item [(2)]$\limsup \mu_\alpha (f)\leq \mu(f)$, for all $f$ bounded upper
semicontinuous function;
\item [(3)]$\liminf \mu_\alpha (f)\geq \mu(f)$, for all $f$ bounded lower
semicontinuous function;
\item [(4)]$\lim_\alpha \mu_\alpha(X)=\mu(X)$ and $\limsup \mu_\alpha (F)\leq \mu(F)$, for all closed set $F\subset X$;
\item [(5)]$\lim_\alpha \mu_\alpha(X)=\mu(X)$ and $\liminf \mu_\alpha (G)\geq \mu(G)$, for all open set $G\subset X$;
\item [(6)]$\lim_\alpha\mu_\alpha(f)=\mu(f)$, for all bounded continuous
function $f$.
\end{itemize}

Then the first five statements are equivalent and each of them implies the last one.

Furthermore, if $X$ is a completely regular space and $\mu\in \mM(X,t)$, then all six statements are equivalent.
\end{lem}

We next state a result of compactness on the space of tight measures $\mM(X,t)$ that is going to be essential for our main result. For a proof of this fact, see \cite[Theorem 9.1]{Topsoebook}.

\begin{thm}\label{reformulated}
Let $X$ be a Hausdorff topological space and let $(\mu_\alpha)_\alpha$ be a net in
$\mM(X,t)$ such that $\limsup\mu_\alpha(X)<\infty$. If $(\mu_\alpha)_\alpha$ is uniformly tight, then it is compact with respect to the weak semi-continuity topology in $\mM(X,t)$.
\end{thm}

The previous theorem allows us to obtain convergent subnets of a given net in $\mM(X,t)$, provided it satisfies the required conditions. An important property of the weak semicontinuous topology in $\mM(X,t)$, which motivated Topsoe to advance his work on it (see \cite[Preface]{Topsoebook}), is that under this topology the space $\mM(X,t)$ is a Hausdorff space, guaranteeing, in particular, the uniqueness of the limits. This has been proved in \cite[Theorem 11.2]{Topsoebook}, but we state and prove it here with more details.

\begin{thm}\label{hausdorff}
Let $X$ be a Hausdorff space. Then, $\mathcal{M}(X,t)$ is a Hausdorff space with respect to the weak semi-continuity topology.
\end{thm}

\begin{proof}
First, recall that a Hausdorff space can be characterized as a topological
space where every convergent net converges to at most one point. Therefore, it is enough
to prove that if $(\mu_\alpha)_\alpha$ is a net in $\mathcal M(X;t)$ which
converges to two different elements $\mu_1,\mu_2 \in \mathcal M(X;t)$,
i.e. $\mu_\alpha \wconv\mu_1$ and $\mu_\alpha \wconv\mu_2$, then
$\mu_1=\mu_2$. For $A\in \mathcal B(X)$, denote by $\mathring A$ the interior
of $A$ and by $\bar A$ the closure of $A$. Using Lemma
\ref{portmanteau}, we obtain that
\[\mu_1(\mathring A)\leq \liminf_\alpha \mu_\alpha(\mathring A)\leq
\limsup_\alpha \mu_\alpha(\bar A)\leq \mu_2(\bar A).\]
Now, for $E\in \mathcal B(X)$, let us prove that $\mu_1(E)\leq \mu_2(E)$. In
order to do so, consider any compact sets $K_1\subset E$ and $K_2\subset E^c$.
Since $X$ is Hausdorff there exist disjoint open sets $A$ and $B$ such that
$K_1\subset A$ and $K_2\subset B$. It is clear that $\bar A\subset
X\setminus K_2$. Thus,
\[\mu_1(K_1)\leq \mu_1(A)\leq \mu_2(\bar A)\leq
\mu_2(X\setminus K_2)=\mu_2(X)-\mu_2(K_2),\]
which leads us to
\[\mu_1(K_1)+\mu_2(K_2)\leq \mu_2(X).\]
Since $K_1$ and $K_2$ are arbitrary compact sets satisfying $K_1\subset E$
and $K_2\subset E^c$, we can take the supremum over all compact
sets $K_1\subset E$ and the supremum  over all compact sets $K_2\subset
E^c$ and we find that
\[\sup\{\mu_1(K_1)\,|\, K_1 \mbox{ compact}, K_1\subset E \}
+\sup\{\mu_2(K_2)\,|\,K_2 \mbox{ compact}, K_2\subset E^c \}\leq \mu_2(X).\]
Since $\mu_1$ and $\mu_2$ are tight, we conclude that
\[\mu_1(E)+\mu_2(E^c)\leq \mu_2(X).\]
Thus $\mu_1(E)\leq \mu_2(E)$, for all $E\in \mathcal B(X)$. Now, since by using Lemma \ref{portmanteau} we have $\mu_1(X)=\mu_2(X)$, it follows that $\mu_1=\mu_2$.
\end{proof}



Evidently, all the results shown above are also valid in the space of probability measures. In the next section, these results are applied in that space, since it is the natural one in the context of statistical solutions. We consider both the spaces of probability measures defined over the Hausdorff space $X$ and over the space of continuous paths $\mX$.

\section{Abstract Results}
\label{secabstract}

The purpose of this section is to provide a general definition of trajectory statistical solutions and to prove their existence in regard to a given initial probability distribution and under suitable hypotheses. These abstract trajectory statistical solutions are defined with respect to a subset $\mU$ of $\mX$, upon which the hypotheses are imposed. The proof of existence is given at the end of the section, after some essential lemmas.

\subsection{Definition of Trajectory Statistical Solutions}
\label{secdefstatsol}

We define below our abstract concept of a trajectory statistical solution.

\begin{defs}\label{def-stat-sol}
Consider a subset $\mU\subset\mX$. We say that a Borel probability measure $\rho$ on $\mX$ is a \textit{$\mU$-trajectory statistical solution} if
\begin{itemize}
\item [(i)] $\rho$ is tight;
\medskip
\item [(ii)] $\exists \mV\in\mathcal B(\mX)$ such that $\mV\subset \mU$ and $\rho(\mV)=1$.
\end{itemize}
\end{defs}

Condition $(ii)$ above can be rephrased by saying that $\rho$ is carried by a Borel subset $\mV$ of $\mU$.

As mentioned in the Introduction, our abstract definition was inspired by the concept of a Vishik-Fursikov measure given in \cite{FRT}, which is carried by the Borel set of Leray-Hopf weak solutions of the Navier-Stokes equations. Since in our case we do not assume a~priori that $\mU$ is a Borel set, we need to define our solutions as being carried by a Borel subset of $\mU$.

Note that whenever $\mU$ is a nonempty set, we can always obtain a $\mU$-trajectory statistical solution by considering the Dirac measure $\delta_{u}$, for any element $u\in\mU$ ($\delta_u$ is tight and $\{u\}$ is a Borel set in $\mU$ satisfying $\delta_{u}(\{u\})=1$). Our main concern thus is not the existence of a trajectory statistical solution itself, but actually the existence of a $\mU$-trajectory statistical solution for an initial value problem, which consists in the following:

\begin{prob}[Initial Value Problem]
  \label{initvalueprob}
  Let $I\subset \mathds{R}$ be an interval closed and bounded on the left, with left end point $t_0$, and let $X$ be a Hausdorff topological space. Let $\mX=\Cloc(I,X)$ be the space of continuous paths in $X$ endowed with the compact open topology. Let $\mU$ be a given subset of $\mX$. Given an ``initial'' tight Borel probability measure $\mu_0\in\mP(X)$ on $X$, we look for a $\mU$-trajectory statistical solution $\rho$ on $\mX$ satisfying $\Pi_{t_0}\rho = \mu_0$, i.e. we look for a measure $\rho\in\mP(\mX)$ satisfying conditions $(i)$ and $(ii)$ of Definition \ref{def-stat-sol} and such that
  \[\rho(\Pi_{t_0}^{-1}(A))=\mu_0(A)\,,\,\,\forall A\in\mB(X).\]
\end{prob}

Although the definition above was given for an arbitrary $\mU$, in order to obtain an existence result, a series of hypotheses must be considered over this set.

Since here we are only interested on the initial value problem, we consider, from now on, an interval $I\subset\mathds{R}$ which is closed and bounded on the left, with left end point $t_0$. The time $t_0$ represents the initial time.

\subsection{Hypotheses on the set of trajectories}

We present below the fundamental set of hypotheses for the abstract framework that allows us to obtain an existence result for the initial value problem described in Problem \ref{initvalueprob}.

\begin{defs}\label{hypothesisH}
We say that a subset $\mathcal{U}\subset\mX$ satisfies the hypothesis $(H)$ if the following conditions are satisfied
\renewcommand{\theenumi}{{H}\arabic{enumi}}

\begin{enumerate}
\item \label{H1} $\Pi_{t_0}\mU=X$;
\item \label{H2} $\Pi_{t_0}^{-1}K\cap\mathcal{U}$ is relatively compact in $\mX$, for every compact subset $K\subset X$;
\item \label{H3} There exists a function $V:I\times\mX\rightarrow \mathbb{R}$ such that
\renewcommand{\theenumii}{\roman{enumii}}
\makeatletter
\renewcommand{\p@enumii}{}
\makeatother
\begin{enumerate}
\item \label{H3i} $V(\cdot,u)$ is lower semi-continuous on $I$, for every $u\in\mX$;
\item \label{H3ii} $V(t_0,\cdot)$ is lower semi-continuous on $\mX$;
\item \label{H3iii} $V(t_0,u)=V(t_0,v)$, for every $u,v\in\mX$ such that $u(t_0)=v(t_0)$;
\item \label{H3iv} For every $t\in I$ and for every compact $K\subset X$, $V(t,\cdot)$ is lower semi-continuous on $\overline{\Pi_{t_0}^{-1}K \cap \mU}$;
\item \label{H3v} For every compact $K\subset X$, $V(t_0,\cdot)$ is bounded on $\Pi_{t_0}^{-1}K\cap\mU$;
\item \label{H3vi} For every compact $K\subset X$,
\begin{eqnarray*}
\Pi_{t_0}^{-1}K\cap\mU&=&\left\{u\in\overline{\Pi_{t_0}^{-1}K\cap\mU}\,|\,\liminf_{t\rightarrow t_0^+}V(t,u)\leq V(t_0,u)\right\}\\
&=&\left\{u\in\overline{\Pi_{t_0}^{-1}K\cap\mU}\,|\,\limsup_{t\rightarrow t_0^+}V(t,u)\leq V(t_0,u)\right\};
\end{eqnarray*}
\item \label{H3vii} For every compact $K\subset X$,
\[\liminf_{t\rightarrow t_0^+}\sup_{u\in\Pi_{t_0}^{-1}K\cap\mU}(V(t,u)-V(t_0,u))\leq 0.\]
\end{enumerate}
\end{enumerate}
\end{defs}

Reading the hypotheses above without having a background equation in mind does not give much insight on how they would fit into a specific problem. So let us imagine for a moment that we want to prove the existence of a statistical solution for a given evolution equation. Then $X$ is taken to be a phase space of the equation, and is assumed that the solutions of the equation are continuous as functions from a time-interval $I$ to $X$. Hence, the solutions belong to the space of continuous paths $\mX=\Cloc(I,X)$. The space $\mU$ is taken to be the set of solutions of the equation with the topology inherited from $\mX$.

The first hypothesis, \eqref{H1}, is simply a mathematical statement of existence of these individual solutions, i.e. given any initial condition $u_0$ in the phase space $X$, there exists a continuous solution $u\in \mU$ with $u(t_0)=u_0$.

Hypothesis \eqref{H2} is usually a consequence of the compactness obtained through typical \textit{a~priori} estimates satisfied by the solutions.

The function $V$ satisfying \eqref{H3} is needed when $\mU$ is not a closed set, as is the case with the Leray-Hopf weak solutions of the Navier-Stokes equations and similar systems. In this case, in order to obtain a Borel set in $\mU$ as a carrier for the statistical solution, we must essentially use the structure of the energy inequality associated with the equation (see \cite{FRT}), but which is not present in an abstract setting. Hypotheses \eqref{H3i} to \eqref{H3vii} of \eqref{H3} basically compensate for the lack of such an explicit inequality.

\subsection{Existence of Trajectory Statistical Solutions}

The present section is dedicated to the proof of existence of a $\mU$-trajectory statistical solution satisfying a given initial condition, as described in the Initial Value Problem \ref{initvalueprob}. We use that the initial probability measure is assumed to be tight to reduce the problem to the case of an initial probability measure carried by a compact set. In order to solve this reduced problem, we need a series of lemmas that we prove below. After proving those lemmas, we address the main theorem.

Lemma \ref{choquet} allows us to approximate the initial measure with a compact support $K$ by convex combinations of Dirac measures satisfying a particular inequality. This inequality involves a real-valued function $V_0$ on $X$ associated with the characterization \eqref{H3vi} of \eqref{H3}, related to the trajectory space $\mU$. This function $V_0$ is introduced in \eqref{defV0} and is proved to be lower-semicontinuous in Lemma \ref{V0lowersemicont}, a condition which is needed when applying Lemma \ref{portmanteau} to the approximating net. Then, we prove that the candidate solution is a measure carried by $\KU$. We verify that $\KU$ is in fact a measurable set in Lemma \ref{KUBorel}.

The first lemma we show considers an arbitrary Borel function $f$ defined on a compact space $K$ and a Borel probability measure $\mu$ on $K$. Assuming that $f$ is bounded below, we construct, through a finite partition of $K$, a net of convex combinations of Dirac measures converging to $\mu$ in the weak semi-continuity topology on $X$. This construction is done in such a way that the integral of $f$ with respect to each element of this net is a lower bound for the integral of $f$ with respect to $\mu$. Due to this constraint in the integral, this result can be viewed as a Krein-Milman theorem with a twist.

This construction is in fact related to the approximation of a Riemann integral by lower sums. However, we cannot simply take the infimum of the function on each subset in the partition since we need points in the space $K$ in order to construct the Dirac measures for the approximation of the initial measure. This would certainly work if the function were continuous since the infimum would be a minimum. Actually, if the function were continuous, we could take a point which coincides with the mean value of the function, and the inequality would become an equality. This is in fact exploited in the theory of Choquet, in particular in \cite[Lemma 26.14]{Choquet2}, to yield a single converging net that preserves the integral value, for arbitrary linear functionals. In our case, however, the function is not continuous, but fortunately the inequality is all we need. For that, the infimum value of $f$ need not be attained, we just need to choose points such that the value of the function at those points are smaller than the average of the function on the corresponding subset of the partition.

\begin{lem}\label{choquet}
Let $K$ be a compact Hausdorff topological space. If $\mu$ is a tight Borel probability measure on $K$ and $f:K\rightarrow\mathbb{R}$ is a Borel function which is bounded below, then there exists a net $(\mu_\alpha)_\alpha$ of discrete Borel probability measures on $K$ such that $\mu_\alpha \wconv \mu$ and
\be\label{eq0}
\int_K f \rd\mu_\alpha\leq \int_K f \rd\mu\,,\,\,\forall \alpha.
\ee
\end{lem}
\begin{proof}
If $f$ assumes negative values, consider the function $g=f-m$, with $m=\inf_{x\in K}f(x)$. Then, once the result is proved for $g$, it can also be obtained for $f$ using that $\mu(X)=\mu_\alpha(X)=1$, for every $\alpha$. So we assume, for simplicity, that $f$ is non-negative.

Let $\mathcal{A}$ be a covering of $K$ by open sets. Since $K$ is compact, there exists $\{A_1,\ldots,A_n\}$ $\subset\mathcal{A}$ such that $K=\bigcup_{i=1}^n A_i$. Furthermore, since every compact Hausdorff space is normal, it follows that there exists a partition of unity $\{g_i\}_{i=1}^n$ subordinated to $\{A_i\}_{i=1}^n$ \cite[Theorem 36.1]{Munkres}. We may assume, without loss of generality, that $\text{supp}\,g_i\neq\emptyset$, for all $i\in\{1,\cdots,n\}$.

Consider each $i\in\{1,\cdots,n\}$. If $\mu(g_i)=0$, then let $x_i$ be any point in $\text{supp}\,g_i$. If $\mu(g_i)\neq0$, define first the measure
\[\nu_i=\frac{g_i\mu}{\mu(g_i)}\,,\]
which acts on a set $A\in \mathcal{B}(K)$ as
\begin{eqnarray*}
\nu_i(A)&=&\frac{1}{\mu(g_i)}\int_K g_i\chi_A \rd\mu\\
        &=&\frac{1}{\mu(g_i)}\int_A g_i \rd\mu.
\end{eqnarray*}

Note that $\nu_i\in \mP(K)$ and that $\nu_i$ is carried by $\text{supp}\,g_i$. Since $g_i$ is continuous and with $\text{supp}\,g_i \neq \emptyset$, then $\text{supp}\,g_i$ has positive measure with respect to $\nu_i$. We claim that there exists $x_i\in \text{supp}\,g_i$ such that
\begin{equation}\label{eq1}
f(x_i)\leq \int_K f \rd\nu_i.
\end{equation}
In fact, if $\int_K f\rd\nu_i=+\infty$, then any $x_i\in \text{supp}\,g_i$ satisfies \eqref{eq1}. On the other hand, let us suppose that $\int_K f\rd\nu_i<+\infty$. If we had
\[f(x)>\int_K f\rd\nu_i\,,\]
for every $x\in \text{supp}\,g_i$, then
\[\int_{supp\,g_i}f\rd\nu_i>\int_{supp\,g_i}\left( \int_K f\rd\nu_i\right)\rd\nu_i=\int_K f\rd\nu_i=\int_{supp\,g_i} f\rd\nu_i,\]
which is a contradiction.
Thus, we may consider points $x_i\in\text{supp}\,g_i\subset A_i$ satisfying \eqref{eq1}, for any $i\in\{1,\cdots,n\}$ such that $\mu(g_i)\neq0$.

Now, define
\[\mu_{\mathcal{A}}=\sum_{i=1}^n\mu(g_i)\delta_{x_i}.\]
Note that $\mu_{\mathcal{A}}\in\mP(K)$. Also, using that $f\geq 0$ and that $\{g_i\}_{i=1}^n$ is a partition of unity, we obtain
\begin{eqnarray*}
\int_Kf\rd\mu_{\mathcal{A}}&=&\sum_{i=1}^n\mu(g_i)f(x_i)=\sum_{\stackrel{1\leq i\leq n}{\mu(g_i)\neq0}}\mu(g_i)f(x_i)\leq\\
&\leq&\sum_{\stackrel{1\leq i\leq n}{\mu(g_i)\neq0}}\mu(g_i)\int_Kf\rd\nu_i=\sum_{\stackrel{1\leq i\leq n}{\mu(g_i)\neq0}}\int_Kg_if\rd\mu\\
&\leq&\sum_{i=1}^n\int_Kg_if\rd\mu=\int_Kf\rd\mu.
\end{eqnarray*}

Considering the set of all open coverings $\mA$ of $K$ ordered by refinement, it follows that $(\mu_{\mathcal{A}})_{\mathcal{A}}$ is a net of discrete Borel probability measures on $K$. Then, it only remains to prove that $\mu_{\mathcal{A}}\wconv \mu$.

Let $\varphi\in\mathcal{C}_{b}(K)$ or, equivalently, let $\varphi$ be a continuous function on $K$.
Since $K$ is compact, we have that $\varphi$ is uniformly continuous. Then, given $\varepsilon>0$, we may choose a covering $\mathcal{A}$ of $K$ such that if $A\in\mathcal{A}$ then
\[|\varphi(x)-\varphi(y)|<\varepsilon\,,\,\,\forall x,y \in A.\]
Note that
\[\left|\int_K\varphi \rd\mu_{\mathcal{A}}-\int_K\varphi \rd\mu\right|\leq\sum_{i=1}^n\left|\mu(g_i)\varphi(x_i)-\int_K g_i\varphi\rd\mu\right|.\]
But
\begin{eqnarray*}
\left|\mu(g_i)\varphi(x_i)-\int_K g_i\varphi\rd\mu\right|&=&\left|\int_K g_i(x)(\varphi(x_i)-\varphi(x))\rd\mu(x)\right|\\
&\leq&\sup_{x\in
supp\,g_i}|\varphi(x_i)-\varphi(x)|\int_K g_i\rd\mu.
\end{eqnarray*}
Thus,
\[\left|\int_K\varphi \rd\mu_{\mathcal{A}}-\int_K\varphi \rd\mu\right|<\left(\sum_{i=1}^n\int_K g_i\rd\mu\right)\varepsilon=\varepsilon\,.\]
This implies that
\[ \int_K \varphi \rd\mu_{\mathcal{A}} \rightarrow \int_K \varphi \;\rd\mu,
\]
for every (bounded) continuous function on $K$, which means that $\mu_{\mathcal{A}}$ converges weak star to $\mu$. Now since $K$ is in particular a completely regular space and $\mu$ is a tight measure, it follows by Lemma \ref{portmanteau} that $\mu_\mA\wconv\mu$.
\end{proof}

Hypothesis \eqref{H1} and condition \eqref{H3iii} of \eqref{H3} imply that for every $u_0$ in $X$ there exists an element $u\in \mX$ such that $u(t_0)=u_0$ and that $V(t_0,u)$ independs of the choice of such $u$. This allows us to define a function $V_0:X\rightarrow\mathds{R}$ given by
\be\label{defV0}V_0(u_0)=V(t_0,u)\,,\,\,\text{for any $u\in\mX$ such that $u(t_0)=u_0$}.\ee
We prove next that the lower semi-continuity of $V(t_0,\cdot)$, guaranteed by condition \eqref{H3ii} of \eqref{H3}, implies that $V_0$ is also lower semi-continuous.

\begin{lem}\label{V0lowersemicont}
$V_0$ is a lower semi-continuous function on X.
\end{lem}
\begin{proof}
Let $u_0\in X$. Consider the function $u\in\Cloc(I,X)$ defined by
\[u(t)=u_0\,,\,\,\forall t\in I.\]
Since $V(t_0,\cdot)$ is a lower semi-continuous function on $\mX$ and the family of sets $\{S(J,U):\,J\subset I\text{ is compact},\,U\subset X\text{ is open}\}$ form a subbase for the compact-open topology in $\Cloc(I,X)$, then for any given $\varepsilon>0$ there exists a neighborhood $B$ of $u$ of the form
\[B=\bigcap_{i=1}^n S(J_i,U_i)\,,\,\,J_i\subset I\text{ compact},\,U_i\subset X\text{ open},\]
such that
\be\label{eq5a}
V(t_0,v)\geq V(t_0,u)-\varepsilon\,,\,\,\forall v\in B.
\ee
Let $U=\bigcap_{i=1}^n U_i$ and consider $a_i$, $b_i$ in $I$ such that $J_i\subset[a_i,b_i]$, for every $i\in\{1,\ldots,n\}$. Set $t_1=\max\{b_1, \ldots, b_n\}$ and let $J=[t_0,t_1]$, where $t_0$ is the left end point of $I$. Note that $u_0\in U$ and, consequently, $u\in S(J,U)$. Since $S(J,U)\subset B$, then \eqref{eq5a} is valid in particular for every $v\in S(J,U)$. By the definition of $V_0$, we then obtain that
\be\label{eq5a1}
V_0(v_0)\geq V_0(u_0)-\varepsilon\,,\,\,\forall v_0\in \Pi_{t_0}S(J,U).
\ee
Now observe that $\Pi_{t_0}S(J,U)=U$. Indeed, since $t_0\in J$ then $\Pi_{t_0}v\in U$, for every $v\in S(J,U)$. On the other hand, for every $v_0\in U$ the function
\[v(t)=v_0\,,\,\,\forall t\in I,\]
is such that $v \in S(J,U)$ and $\Pi_{t_0}v=v_0$.

Thus, since $U$ is a neighborhood of $u_0$, \eqref{eq5a1} implies that $V_0$ is lower semi-continuous on $u_0$.
\end{proof}

\begin{rmk}
If $X$ is a locally path connected space, then an alternative proof can be given for Lemma \ref{V0lowersemicont}. First, we observe that, for every $c\in\mathds{R}$, we can write
\[V_0^{-1}((c,\infty))=\Pi_{t_0}\{V(t_0,\cdot)^{-1}((c,\infty))\}.\]
Since by \eqref{H3ii} the function $V(t_0,\cdot)$ is lower semi-continuous on $\mX$,
we have that the set $V(t_0,\cdot)^{-1}((c,\infty))$ is open in $\mX$. Now, using that $X$ is locally path connected, it can be showed that $\Pi_{t_0}$ is an open map. Then, it follows that $V_0^{-1}((c,\infty))$ is an open set in $X$.

 When $X$ is not locally path connected, however, the map $\Pi_{t_0}$ is not necessarily open, as the following example shows: Let $A\subset\mathds{R}^2$ be the set
\[A=\{(x,\sin(1/x)):\,x\in(0,1]\}.\]
Note that the closure of $A$ in $\mathds{R}^2$ is the set $(\{0\}\times[-1,1])\cup A$. Now consider the space $X$ as
\[X=(\{0\}\times[-2,2])\cup A,\]
with the topology inherited from $\mathds{R}^2$. This space is compact and connected but not locally path connected. If $(0,y)\in X$ is such that $y\notin [-1,1]$, there exists a neighborhood $U$ of $(0,y)$ in $X$ such that $U\cap A=\emptyset$. Then, for any compact subinterval $J\subset I$ with $t_0\notin J$, it follows that
\[\Pi_{t_0}(S(J,U))=\{0\}\times[-2,2],\]
where $S(J,U)$ is the open set in the compact open-topology of $\mX = \Cloc(I,X)$ defined in \eqref{defSJU}. Since $\{0\}\times[-2,2]$ is not open in $X$, it follows that $\Pi_{t_0}$ cannot be an open map.
\end{rmk}

The next result is needed for measurability purposes, as explained in the beginning of this section.

\begin{lem}\label{KUBorel}
For every compact subset $K\subset X$, $\KU$ is a Borel set in $\mX$.
\end{lem}

\begin{proof}
It follows from the characterization of $\KU$ with the $\limsup$ in condition \eqref{H3vi} of \eqref{H3} that
\[\KU=\bigcap_n\bigcup_m B_{n,m},\]
where
\[B_{n,m}=\left\{u\in\fKU\,\left|\,V(t,u)\leq V(t_0,u)+\frac{1}{n},\,\forall t\in\left(t_0,t_0+\frac{1}{m}\right)\right.\right\}.\]
Then, using \eqref{H3i} of \eqref{H3}, each set $B_{n,m}$ can be written as
\[B_{n,m}=\bigcap_{q\in\mathds{Q}\cap\left(t_0,t_0+\frac{1}{m}\right)}\left\{u\in\fKU\,\left|\,V(q,u)\leq V(t_0,u)+\frac{1}{n}\right.\right\},\]
where $\mathds{Q}$ denotes the set of rational numbers. But since \eqref{H3ii} and \eqref{H3iv} imply that $V(t,\cdot)-V(t_0,\cdot)$ is a Borel function in $\fKU$, for every $t\in I$, we obtain that each set inside the intersection above is a Borel set in $\fKU$ and, consequently, $B_{n,m}$ is Borel in $\fKU$. Moreover, since $\fKU$ is a Borel set in $\mX$, then $B_{n,m}$ is also Borel in $\mX$. Thus, $\KU$ is a Borel set in $\mX$.
\end{proof}

We are now in a position to prove our main theorem, concerning the existence of $\mU$-trajectory statistical solutions, as given in Definition \ref{def-stat-sol}. Let us outline the main ideas of the proof.

Starting with an initial tight measure $\mu_0$, at a given time $t_0$, our intention is to show the existence of a measure $\rho$ which is a $\mU$-trajectory statistical solution satisfying the initial condition $\Pi_{t_0}\rho = \mu_0$. As in classical methods, this measure $\rho$ is obtained from the limit of a convergent net of measures.

We first consider the case when the initial measure $\mu_0$ is carried by a compact set. Then, using Lemma \ref{choquet}, we obtain a net $(\mu_0^\alpha)_\alpha$ of discrete measures converging to $\mu_0$ in the phase space and also satisfying inequality \eqref{eq0} with respect to the function $V_0$ defined in \eqref{defV0}. Using hypothesis \eqref{H1}, we can easily extend each discrete initial measure $\mu_0^\alpha$ to a discrete measure $\rho_\alpha$ in the trajectory space, by applying \eqref{H1} to each point in the support of $\mu_0^\alpha$. By construction, each $\rho_\alpha$ is a tight measure carried by $\fKU$, which by our hypothesis \eqref{H2} is a compact set. This implies that $(\rho_\alpha)_\alpha$ is a uniformly tight net and then Theorem \ref{reformulated} is applied to obtain a subnet converging to some tight measure $\rho$, also carried by $\fKU$. The fact that $\rho$ satisfies the initial condition, i.e. $\Pi_{t_0}\rho=\mu_0$, follows easily from the uniqueness of the limits in $\mM(X,t)$, guaranteed by Theorem \ref{hausdorff}. Then, in order to conclude that $\rho$ is a $\mU$-trajectory statistical solution, it remains to show that it is carried by a Borel subset of $\mU$. This is done by using the characterization of $\KU$ inside $\fKU$, given by hypothesis \eqref{H3vi} of \eqref{H3}, yielding that $\rho$ is in fact carried by $\KU$. The inequality \eqref{eq0}, enforced in the construction of the net $(\mu_0^\alpha)_\alpha$ is used precisely here.

The proof of the case when $\mu_0$ is not carried by any compact set, can be reduced to the previous case by using the hypothesis that $\mu_0$ is a tight measure. The idea consists in decomposing $\mu_0$ as a sum of tight measures, each being carried by a compact set. The previous case can then be applied to each of these tight measures, yielding a countable family of $\mU$-trajectory statistical solutions. Our desired measure is then obtained as an appropriate weighted sum of these particular measures.

There are some technical details that we skipped in the previous discussion and which are concerned with the restriction of the approximating measures to convenient compact subspaces. If we assumed that our underlying phase space was completely regular, the proof could be made a bit simpler, as these restrictions would no longer be necessary. But again, looking for a higher degree of generality, we assume only that our phase space $X$ is a Hausdorff space.

\begin{thm}\label{existence}
Let $X$ be a Hausdorff topological space and let $I$ be a real interval closed and bounded on the left with left end point $t_0$. If $\mU\subset \mX$ is a subset satisfying hypothesis $(H)$ then for any tight Borel probability measure $\mu_0$ on $X$ there exists a $\mU$-trajectory statistical solution $\rho$ on $\mX$ such that $\Pi_{t_0}\rho=\mu_0$.
\end{thm}
\begin{proof}
Let us first suppose that $\mu_0$ is carried by a compact subset $K\subset X$.

Consider the function $V_0:X\rightarrow \mathbb {R}$ defined by \eqref{defV0}. Lemma \ref{V0lowersemicont} implies that $V_0$ is a Borel function on $X$ and hence the restriction $V_0|_K$ to $K$ is a Borel function on $K$. Moreover, by condition \eqref{H3v} of \eqref{H3} we have that $V_0|_K$ is bounded on $K$, in particular, bounded from below (of course, this also follows from the fact that $V_0$ is lower semi-continuous on the compact set $K$). Then, applying Lemma \ref{choquet} to the compact space $K$ and with $V_0|_K$ in the place of $f$, we obtain a net $(\mu_0^\alpha)_\alpha$ of discrete measures in $\mP(K)$ such that $\mu_0^\alpha \wconv\mu_0|_K$ and
\be\label{eq5b}
\int_K V_0|_K(u_0)\rd \mu_0^\alpha(u_0)\leq\int_K V_0|_K(u_0)\rd\mu_0|_{K}(u_0),\quad\forall \alpha.
\ee
Since $(\mu_0^\alpha)_\alpha$ is a net of discrete measures then, for each $\alpha$, there exist $J_\alpha\in\mathbb{N}$, $\theta_j^\alpha\in\mathbb{R}$ and $u_{0,j}^\alpha\in K$ such that
\[0<\theta_j^{\alpha}\leq 1\,\,,\,\,\,\sum_{j=1}^{J_\alpha}\theta_j^{\alpha}=1,\]
and
\[\mu_0^\alpha=\sum_{j=1}^{J_\alpha}\theta_j^\alpha\delta_{u_{0,j}^\alpha}.\]
Furthermore, from \eqref{H1} it follows that for each $u_{0,j}^\alpha$ there exists $u_j^\alpha\in\mathcal{U}$ such that $\Pi_{t_0}u_j^\alpha=u_{0,j}^\alpha$. Then, let $\rho_\alpha$ be the measure in $X$ defined by
\[\rho_\alpha=\sum_{j=1}^{J_\alpha}\theta_j^\alpha\delta_{u_j^\alpha}.\]

Note that $\rho_\alpha$ belongs to $\mP(\mX,t)$ and is carried by $\KU$. Using that $\fKU$ is compact (see \eqref{H2}) and taking the restriction of each $\rho_\alpha$ to this set, we obtain that $(\rho_\alpha|_{\fKU})_\alpha$ is a net in $\mP(\fKU,t)$ which is clearly uniformly tight.\comment{If $\rho$ is a tight Borel measure on a topological space $X$, then its restriction $\rho|_A$ to any Borel set $A\in \mB(X)$, defined as
\[\rho|_A(B)=\rho(B\cap A)\,,\,\,\forall B\in\mB(A),\]
is also a tight measure on $A$, with respect to the induced topology. Indeed, consider $B\in\mB(A)$ and $\varepsilon>0$. Then $B$ can be written as $B'\cap A$ for $B'\in\mB(X)$, which implies that $B$ is also a Borel set in $X$. Therefore, by the tightness of $\rho$ in $X$ it follows that there exists a compact set $K$ of $X$ such that $K\subset B$ and $\rho(B)-\rho(K)<\varepsilon$. But
\[\rho|_A(B)-\rho|_A(K)=\rho(B)-\rho(K)<\varepsilon\]
Therefore, $\rho|_A$ is a tight Borel measure on $A$.} Then, by Theorem \ref{reformulated} there is a measure $\rho$ in $\mP(\fKU,t)$ such that, by passing to a subnet if necessary,
\be\label{eq6}
\rho_\alpha|_{\fKU}\wconv \rho\mbox{ in }\fKU.
\ee

Then, defining $\tilde{\rho}\in\mM(\mX)$ as
\[\tilde{\rho}(A)=\rho(A\cap\fKU)\,,\,\,\forall A \in\mB(\mX),\]
it is not difficult to see that $\tilde{\rho}\in\mP(\mX,t)$ and $\rho_\alpha\wconv\tilde{\rho}$ in $\mX$.
\comment{If $X$ is a topological space and $\rho$ is a tight Borel measure defined on a Borel subset $A$ of $X$, then its extension $\tilde{\rho}$ to $X$, defined by
\[\tilde{\rho}(B)=\rho(B\cap A)\,,\,\,\forall B\in\mB(X),\]
is a tight Borel measure on $X$. Indeed, consider $B\in\mB(X)$ and $\varepsilon>0$. Then, since $B\cap A\in\mB(X)$, from the tightness of $\rho$ on $\mB(A)$ it follows that there exists a set $K\subset B\cap A\subset B$ which is compact in $A$ and such that $\rho(B\cap A) - \rho(K)<\varepsilon$. Thus,
\[\tilde{\rho}(B)-\tilde{\rho}(K)=\rho(B\cap A)-\rho(K\cap A)=\rho(B\cap A)-\rho(K)<\varepsilon.\] Since $K$ is also compact in $X$, this means that $\tilde\rho$ is tight.}\comment{Moreover, if $(\rho_\alpha)_\alpha$ is a net in $\mM(\mX)$ and there exists a compact set $\mK\subset\mX$ and a measure $\rho\in\mM(\mK)$ such that $\rho_\alpha|_\mK\wconv \rho$ in $\mK$ and each $\rho_\alpha$ is carried by $\mK$, then $\rho_\alpha\wconv \tilde{\rho}$ in $\mX$. In order to prove this, consider a bounded and upper semi-continuous function $\varphi:\mX\rightarrow\mathds{R}$. Then $\varphi|_\mK$ is evidently bounded on $\mK$. Furthermore, given $u\in\mK$ and $\varepsilon>0$, since $\varphi$ is upper semi-continuous on $\mX$, there exists a neighborhood $\mW$ of $u$ in $\mX$ such that
\[\varphi(v)\leq\varphi(u)+\varepsilon\,,\,\,\forall v\in \mW,\]
which implies that
\[\varphi|_\mK(v)\leq\varphi|_\mK(u)+\varepsilon\,,\,\,\forall v\in \mW\cap\mK.\]
Since $\mW\cap\mK$ is a neighborhood of $u$ in $\mK$, then $\varphi|_\mK$ is also upper semi-continuous on $\mK$. Thus, since $\rho_\alpha$ and $\tilde{\rho}$ are carried by $\mK$, by Lemma \ref{portmanteau} we obtain that
\[\limsup_\alpha\int_\mX\varphi\rd \rho_\alpha=\limsup_\alpha\int_\mK\varphi|_\mK\rd \rho_\alpha|_\mK\]
\[\leq\int_\mK\varphi|_\mK\rd\rho=\int_\mX\varphi\rd\tilde{\rho}.\]
By Lemma \ref{portmanteau} again, we conclude that $\rho_\alpha\wconv\tilde{\rho}$ in $\mX$.}

Since $\Pi_{t_0}:\mX\rightarrow X$ is continuous, using \eqref{eq00} and Lemma \ref{portmanteau} one then obtains that $\Pi_{t_0}\rho_\alpha\wconv\Pi_{t_0}\tilde{\rho}$ in $X$. Moreover, taking the restrictions of these measures to the compact $K$, we also obtain that $\Pi_{t_0}\rho_\alpha|_K\wconv\Pi_{t_0}\tilde{\rho}|_K$. On the other hand, we have by construction that
\[\Pi_{t_0}\rho_\alpha|_K=\mu_0^\alpha\wconv\mu_0|_K.\]
Adding this to the fact that $\Pi_{t_0}\tilde{\rho}|_K,\mu_0|_K\in\mP(K,t)$, by Theorem \ref{hausdorff} we obtain that $\Pi_{t_0}\tilde{\rho}|_K=\mu_0|_K$. But since $\Pi_{t_0}\tilde{\rho}$ and $\mu_0$ are carried by $K$ we then get that $\Pi_{t_0}\tilde{\rho}=\mu_0$.

Furthermore, due to hypothesis \eqref{H3}, we can show that $\tilde{\rho}$ is carried by $\Pi_{t_0}^{-1}K\cap\mU$. Indeed, consider $\varepsilon>0$. Then \eqref{H3vii} implies the existence of a sequence $\{t_j\}$ in $I$ such that $t_j\rightarrow t_0^+$ and
\be\label{eq20}\sup_{u\in\Pi_{t_0}^{-1}K\cap\mU}(V(t_j,u)-V(t_0,u))\leq\varepsilon\,,\,\,\forall j.\ee

Note that
\begin{eqnarray}\label{eq21}
\int_{\fKU}\liminf_{t\rightarrow t_0^+}V(t,u)\rd\tilde{\rho}(u)&\leq&\int_{\fKU}\liminf_j V(t_j,u)\rd\tilde{\rho}(u)\nonumber\\
&\leq&\liminf_j\int_{\fKU}V(t_j,u)\rd\tilde{\rho}(u),
\end{eqnarray}
where the second inequality follows by Fatou's Lemma.

By \eqref{H3v} and \eqref{eq20}, it follows that $V(t_j,\cdot)$ is a function bounded from above on $\KU$.  Since $V(t_j,\cdot)$ is lower semi-continuous on $\mX$ (hypothesis \eqref{H3iv} of \eqref{H3}), this boundedness from above extends to the closure $\fKU$. Moreover, the lower semi-continuity of $V(t_j,\cdot)$ on $\fKU$ and the fact that $\fKU$ is a compact set (hypothesis \eqref{H2}), imply that $V(t_j,\cdot)$ is also bounded from below on $\fKU$. Then, since $\rho_\alpha\wconv\tilde{\rho}$, it follows by Lemma \ref{portmanteau} that
\be\label{eq22}\int_{\fKU}V(t_j,u)\rd\tilde{\rho}(u)\leq\liminf_{\alpha}\int_{\KU}V(t_j,u)\rd\rho_\alpha(u),\ee
where in the right hand side of the inequality above we have used the fact that $\rho_\alpha$ is carried by $\KU$.
Now by \eqref{eq20}, we have that
\begin{multline}\label{eq23}
\int_{\KU}V(t_j,u)\rd\rho_\alpha(u) \\
  \leq \int_{\KU}(V(t_0,u)+\varepsilon)\rd\rho_\alpha(u)
=\int_{\KU}V_0|_K(\Pi_{t_0}u)\rd\rho_\alpha(u)+\varepsilon.
\end{multline}
By \eqref{eq00}, we obtain that
\be\label{eq24}\int_{\KU}V_0|_K(\Pi_{t_0}u)\rd\rho_\alpha(u)=\int_{K}V_0|_K(u_0)\rd\mu_0^\alpha(u_0).\ee
Since the right hand side of \eqref{eq24} does not depend on $j$, we obtain, putting together the inequalities \eqref{eq21} to \eqref{eq24}, that
\be
\label{eq1000}
\int_{\fKU}\liminf_{t\rightarrow t_0^+}V(t,u)\rd\tilde{\rho}(u) \leq \liminf_{\alpha} \int_{K}V_0|_K(u_0)\rd\mu_0^\alpha(u_0) + \varepsilon.
\ee

Now, from \eqref{eq5b}, we have
\be\label{eq25}\liminf_\alpha\int_{K}V_0|_K(u_0)\rd\mu_0^\alpha(u_0)\leq\int_{K}V_0|_K(u_0)\rd\mu_0|_{K}(u_0)=\int_{X}V_0(u_0)\rd\mu_0(u_0),\ee
where the equality follows from the fact that $\mu_0$ is carried by $K$.

But
\begin{eqnarray}\label{eq4}
\int_{X}V_0(u_0)\rd\mu_0(u_0)&=&\int_XV_0(u_0)\rd\Pi_{t_0}\tilde{\rho}(u_0)\nonumber\\
&=&\int_{\mX}V_0(\Pi_{t_0}u)\rd\tilde{\rho}(u)\nonumber\\
&=&\int_{\fKU}V(t_0,u)\rd\tilde{\rho}(u).
\end{eqnarray}
Then, from \eqref{eq1000} to \eqref{eq4}, we conclude that
\[\int_{\fKU}\liminf_{t\rightarrow t_0^+}V(t,u)\rd\tilde{\rho}(u)\leq\int_{\fKU}V(t_0,u)\rd\tilde{\rho}(u)+\varepsilon.\]
Therefore, since $\varepsilon$ is an arbitrary positive number, we find that
\begin{equation}\label{eq5}
\int_{\fKU}\liminf_{t\rightarrow t_0^+}[V(t,u)-V(t_0,u)]\rd\tilde{\rho}(u)\leq 0.
\end{equation}
On the other hand, since $V(\cdot,u)$ is lower semi-continuous at $t_0$, for every $u\in\fKU$, then
\be\label{eq26}\liminf_{t\rightarrow t_0^+}[V(t,u)-V(t_0,u)]\geq 0.\ee
Hence \eqref{eq5} and \eqref{eq26} imply that
\[\liminf_{t\rightarrow t_0^+}[V(t,u)-V(t_0,u)]= 0\,\,\,\,\tilde{\rho}\text{-a.e. on }\fKU.\]
By the characterization with the $\liminf$ in \eqref{H3vi} of \eqref{H3}, we then conclude that $\tilde{\rho}$ is carried by the Borel set $\KU$ (see Lemma \ref{KUBorel}). Then, we have proved that $\tilde{\rho}$ is a $\mU$-trajectory statistical solution with initial condition $\Pi_{t_0}\tilde{\rho}=\mu_0$, in the particular case that $\mu_0$ is carried by a compact subset $K\subset X$.

Now let us prove the case when $\mu_0$ is not carried by any compact subset of $X$. In this case, since $\mu_0$ is a tight measure, there exists a sequence $\{K_n\}$ of compact subsets of $X$ such that
\[\mu_0(K_{n+1})>\mu_0(K_n)>0\,,\,\,\forall n\in\mathds{N},\]
and
\begin{equation}\label{eq3}
\mu_0(X\backslash K_n)<\frac{1}{n}\,,\,\,\forall n\in\mathds{N}.
\end{equation}
Moreover, we may assume that $K_n\subset K_{n+1}$, for all $n$.

Let $D_1=K_1$ and $D_n=K_n\backslash K_{n-1}$, for every $n\geq 2$. Note that
\[\mu_0\left(X\backslash\bigcup_j D_j\right)=\mu_0\left(X\backslash\bigcup_j K_j\right)\leq\mu_0(X\backslash K_n)<\frac{1}{n},\]
for all $n\in\mathbb{N}$. Thus, taking the limit as $n\rightarrow\infty$ above, we obtain that $\mu_0$ is carried by $\bigcup_j D_j$. Then, for every $A\in\mathcal{B}(X)$, since the sets $D_j$, $j\in \mathds{N}$, are pairwise disjoint, we have
\[\mu_0(A)=\mu_0\left(A\cap\left(\bigcup_j D_j\right)\right)=\sum_{j=1}^\infty\mu_0(A\cap D_j).\]

So we may decompose $\mu_0$ as
\[\mu_0=\sum_j\mu_0(D_j)\mu_0^j,\]
where $\mu_0^j$ is the Borel probability measure defined as
\[\mu_0^j(A)=\frac{\mu_0(A\cap D_j)}{\mu_0(D_j)}\,,\,\,\forall A\in\mathcal{B}(X).\]
Note that each $\mu_0^j$ is well-defined, since $\mu_0(D_1)=\mu_0(K_1)>0$ and
\[\mu_0(D_j)=\mu_0(K_j)-\mu_0(K_{j-1})>0\,,\,\,\forall j\geq 2.\]

Also, since each $\mu_0^j$ is carried by the compact set $K_j$, using the first part of the proof, we obtain a tight Borel probability measure $\rho_j$ carried by $\Pi_{t_0}^{-1}K_j\cap\mathcal{U}$ and such that $\Pi_{t_0}\rho_j=\mu_0^j$.

Let $\rho$ be the Borel probability measure defined by
\[\rho=\sum_j\mu_0(D_j)\rho_j.\]

Observe that
\[\rho\left(\bigcup_l \Pi_{t_0}^{-1}K_l\cap\mathcal{U}\right)=\sum_j\mu_0(D_j)\rho_j(\Pi_{t_0}^{-1}K_j\cap\mathcal{U})=\sum_j\mu_0(D_j)=1,\]
where the first equality follows from the fact that $\rho_j$ is carried by $\Pi_{t_0}^{-1}K_j\cap\mU$. Thus, $\rho$ is carried by $\bigcup_j \Pi_{t_0}^{-1}K_j\cap\mathcal{U}$, which is a Borel set in $\mX$ and is contained in $\mU$. The fact that $\Pi_{t_0}\rho=\mu_0$ is also easily verified.

It only remains to show that $\rho$ is a tight measure. In order to prove so, consider a Borel set $\mA\in\mathcal{B}(\mX)$ and $\varepsilon>0$. Let $n\in\mathbb{N}$ be such that $1/n < \varepsilon/2$. Then, since $\rho_j$ is a tight measure, for each $1\leq j\leq n$ there exists a compact set $\mK_j^n\subset \mA$ such that
\[\rho_j(\mA\backslash \mK_j^n)<\frac{\varepsilon}{2n}.\]
Let $\mK^n=\bigcup_{1\leq j\leq n}\mK_j^n$. Note that
\begin{eqnarray*}
\rho(\mA\backslash \mK^n)&=&\sum_{j=1}^\infty\mu_0(D_j)\rho_j(\mA\backslash \mK^n)\\
&\leq&\sum_{j=1}^n\rho_j(\mA\backslash \mK^n)+\sum_{j=n+1}^\infty\mu_0(D_j)\\
&<&\frac{\varepsilon}{2}+\mu_0(X\backslash K_n).
\end{eqnarray*}
Thus, according to \eqref{eq3} and the choice of $n$, it follows that $\rho(\mA\backslash \mK^n)<\varepsilon$. Since $\mK^n$ is a compact set in $\mX$, this proves that $\rho$ is tight.
\end{proof}

\begin{rmk}
Notice that given an initial tight Borel probability measure $\mu_0$ on $X$, if $\mu_0$ is carried by a compact set $K$ on $X$, then the $\mU$-trajectory statistical solution $\rho$ with $\Pi_{t_0}\rho = \mu_0$ obtained in the proof of Theorem \ref{existence} is carried by the Borel set $\Pi_{t_0}^{-1} K \cap \mU$. If $\mu_0$ is not carried by any compact set, then given any sequence of compact subsets $K_n$ of $X$, $n\in \mathds{N}$, such that $\mu(X\setminus K_n) \rightarrow 0$, as $n\rightarrow \infty$, a $\mU$-trajectory statistical solution $\rho$ with $\Pi_{t_0}\rho = \mu_0$ can be constructed such that it is carried by the Borel set $\mU \cap (\bigcup_n \Pi_{t_0}^{-1} K_n)$.
\end{rmk}

Although we only consider in this work the concept of trajectory statistical solutions, which concerns a measure $\rho$ defined over the trajectory space $\mC_{loc}(I,X)$, it is also of interest to develop an abstract definition of time-dependent statistical solutions, generalizing the results obtained in \cite{FRT2010,FRT}. This latter notion considers a family of measures defined over the phase space $X$ and parametrized by the time variable $t$. One such family of measures is easily constructed by projecting the measure $\rho$ on the phase space, at each time $t$, to yield the measures $\mu_t = \Pi_t\rho$ defined on $X$. A question that naturally arises is whether this family of projections is an abstract time-dependent statistical solution, in some suitable sense. The family $\{\mu_t\}_{t\in I}$ has many of the properties one expects from a notion of statistical solution in phase space. The key step in connecting it to an evolution equation, however, is an equation for the moments
\[ \frac{\rd}{\rd t} \int_X \varphi(u) \;\rd\mu_t(u),
\]
which involves an equation for the distribution $\partial_t u$. This needs in particular that the space $X$ be a topological vector space, or a subset of such a space. Moreover, for a sufficiently general result, minimal conditions must be imposed on the regularity of the evolution equation. These results are currently under investigation and will be presented in a future work.

\section{Application to the B\'enard Problem}
\label{secapplication}

The formulation in Section \ref{secabstract} is applicable to a wide range of equations related to incompressible viscous fluids, starting with the three-dimensional incompressible Navier-Stokes equations, and extending to the equations of magnetohydrodynamics, thermohydraulics, geofluidynamics, and other systems coupled with the Navier-Stokes equations.

In order to illustrate the applicability of the formulation, we consider, in this section, the B\'enard problem, which models a phenomenon of convection in fluids, and consisting of the Navier-Stokes equations coupled with an equation for the temperature via the Boussinesq approximation \cite{batchelor,lesieur}.

We shall analyze the three-dimensional case for a homogeneous and incompressible fluid in the region $\{(x_1,x_2,x_3)\in\mathds{R}^3\,|\,0<x_3<h\}$. At the lower surface $x_3=0$, the fluid is heated at a constant temperature $T_0$, while at the upper surface $x_3=h$, the fluid is at a temperature $T_1<T_0$, also constant. Let $\{\bfe_1,\bfe_2,\bfe_3\}$ be the canonical orthonormal basis in $\mathds{R}^3$. Then, through the Boussinesq approximation one obtains the following equations describing the evolution of the velocity field $\bu$, the pressure $p$ and the temperature $T$:

\be\label{eq27}\frac{\partial\bu}{\partial t}+(\bu\cdot\nabla)\bu-\nu\Delta\bu+\nabla p=g\alpha (T-T_1)\bfe_3,\ee
\be\label{eq28}\frac{\partial T}{\partial t}+(\bu\cdot\nabla)T-\kappa\Delta T=0,\ee
\be\label{eq29}\nabla\cdot\bu=0,\ee
where $g$ is the acceleration of gravity, $\alpha$ is the volume-expansion coefficient of the fluid, $\nu$ is the kinematic viscosity and $\kappa$ is the coefficient of thermometric conductivity.

We also consider zero velocity field at the boundaries $x_3=0$ and $x_3=h$ and periodic boundary conditions in the directions $x_1$ and $x_2$, so that the boundary conditions for problem \eqref{eq27}-\eqref{eq29} are given as

\be\label{eq30}\bu=0\quad \text{at} \quad x_3=0\quad \text{and} \quad x_3=h,\ee
\be\label{eq31}T=T_0\quad \text{at} \quad x_3=0,\quad T=T_1\quad \text{at}\quad x_3=h, \ee
\be\label{eq32}p,\bu,T\quad \text{are periodic in the $x_1$ and $x_2$ directions},\ee
where the last condition means that
\[\psi(x_1+L_1,x_2,x_3)=\psi(x_1,x_2,x_3),\quad\forall (x_1,x_2,x_3)\in \mathds{R}^2\times(0,h)\]
\[\psi(x_1,x_2+L_2,x_3)=\psi(x_1,x_2,x_3),\quad\forall (x_1,x_2,x_3)\in \mathds{R}^2\times(0,h)\]
for some positive real numbers $L_1$ and $L_2$, and $\psi$ being any of the functions in condition \eqref{eq32}.

In order to simplify the analysis of the problem, we define a background temperature $T_{b,\varepsilon}$, given by
\[T_{b,\varepsilon}(x_1,x_2,x_3)=\begin{cases}0, & \text{for }0\leq x_3 <h-\varepsilon\medskip\\
                          \displaystyle\frac{(T_1-T_0)}{\varepsilon}(x_3-h+\varepsilon), & \text{for }h-\varepsilon\leq x_3\leq h,\end{cases}\]
where $\varepsilon$ is a positive real number which is chosen appropriately later. Then, we introduce a change of variables for the temperature by considering $\theta=T-T_0-T_{b,\varepsilon}$, in terms of which, problem \eqref{eq27}-\eqref{eq29} is rewritten as

\be\label{eq33}\frac{\partial\bu}{\partial t}+(\bu\cdot\nabla)\bu-\nu\Delta\bu+\nabla p=g\alpha\theta\bfe_3+g\alpha (T_{b,\varepsilon}+T_0-T_1)\bfe_3,\ee
\be\label{eq34}\frac{\partial\theta}{\partial t}+(\bu\cdot\nabla)\theta-\kappa\Delta\theta=-(\bu\cdot\nabla)T_{b,\varepsilon}+\kappa\Delta T_{b,\varepsilon},\ee
\be\label{eq35}\nabla\cdot\bu=0,\ee
with the following boundary conditions
\[\bu=0\quad \text{at} \quad x_3=0\quad \text{and} \quad x_3=h,\]
\[\theta=0\quad \text{at} \quad x_3=0\quad \text{and} \quad x_3=h,\]
\[p,\bu,\theta\quad \text{are periodic in the $x_1$ and $x_2$ directions}.\]
Note that the boundary conditions for the temperature in the $x_3$ direction are now zero, justifying the introduction of this new variable. In fact, there are many possible choices for this background temperature. The one we use here has been chosen so as to yield uniform in time \textit{a priori} estimates.

Let us now introduce the function spaces which are necessary in the following analysis. Consider the domain $\Omega=(0,L_1)\times(0,L_2)\times(0,h)$ and define
\[\mV_1=\{\bv=\bw|_\Omega\,:\,\bw\in(\mD(\mathds{R}^2\times(0,h)))^3,\,\nabla\cdot \bw=0,\,\]
\[\bw\text{ is $L_1$-periodic in the $x_1$ direction and $L_2$-periodic in the $x_2$ direction}\},\]
and
\[\mV_2=\{\theta=\phi|_\Omega\,:\,\phi\in\mD(\mathds{R}^2\times(0,h)),\,\text{$\phi$ is $L_1$-periodic in the $x_1$ direction and}\]
\[\text{$L_2$-periodic in the $x_2$ direction}\},\]
where $\mD(\mathds{R}^2\times(0,h))$ denotes the set of infinitely differentiable and compactly supported functions in $\mathds{R}^2\times(0,h)$.

Then, let $V_1$ be the closure of $\mV_1$ with respect to the $(H_0^1(\Omega))^3$ norm and $H_1$ be the closure of $\mV$ with respect to the $(L^2(\Omega))^3$ norm. Also, let $V_2$ be the closure of $\mV_2$ with respect to the $H_0^1(\Omega)$ norm and $H_2$ be the closure of $\mV_2$ with respect to the $L^2(\Omega)$ norm. We then define the Hilbert spaces $V=V_1\times V_2$ and $H=H_1\times H_2$.

The inner product and norm in $V_1$ are defined as
\[(\!(\bv,\tilde{\bv})\!)_1=\sum_{i,j=1}^3\int_\Omega\nabla v_i\cdot\nabla \tilde{v}_j\rd x,\quad \forall \bv,\tilde{\bv}\in V_1,\]
\[\|\bv\|_1=(\!(\bv,\bv)\!)_1^{1/2},\quad \forall\bv\in V_1.\]
Similarly for $V_2$,
\[(\!(\theta,\tilde{\theta})\!)_2=\int_\Omega \nabla\theta\cdot\nabla\tilde{\theta}\rd x,\quad \forall \theta,\tilde{\theta}\in V_2,\]
\[\|\theta\|_2=(\!(\theta,\theta)\!)_2^{1/2},\quad \forall\theta \in V_2.\]
Then, we define the following inner product and norm in the product space $V=V_1\times V_2$:
\[(\!(z,\tilde{z})\!)_V=(\!(\bv,\tilde{\bv})\!)_1+\gamma(\!(\theta,\tilde{\theta})\!)_2,\quad\forall z=(\bv,\theta),\,\,\tilde{z}=(\tilde{\bv},\tilde{\theta})\in V,\]
\[\|z\|_V=(\!(z,z)\!)_V^{1/2},\quad \forall z\in V,\]
where $\gamma$ is a positive parameter making the above definition dimensionally correct. Like $\varepsilon$, the parameter $\gamma$ is chosen appropriately later.

Similarly, the inner products and norms of the spaces $H_1$ and $H_2$ are the usual ones from $(L^2(\Omega))^3$ and $L^2(\Omega)$, and are denoted respectively by $(\cdot,\cdot)_1$ and $(\cdot,\cdot)_2$, with norms $|\cdot|_1$ and $|\cdot|_2$. The inner product and norm in the space $H$ are then defined accordingly:
\[(z,\tilde{z})_H=(\bv,\tilde{\bv})_1+\gamma(\theta,\tilde{\theta})_2,\quad\forall z=(\bv,\theta),\,\,\tilde{z}=(\tilde{\bv},\tilde{\theta})\in H,\]
\[|z|_H=(z,z)_H^{1/2},\quad \forall z\in H.\]

We identify $H_1$ and $H_2$ with their respective duals and consider the dual spaces $V_1'$ and $V_2'$ of $V_1$ and $V_2$, respectively, so that $V_1\subset H_1=H_1'\subset V_1'$ and $V_2\subset H_2=H_2' \subset V_2'$, with continuous and dense injections.

In the product space, we characterize the dual of $V$ as the space $V'=V_1'\times V_2'$, with the duality product between $h=(\f,g)\in V'$ and $z=(\bu,\theta)\in V$ given by
\[ \langle h,z \rangle = \langle \f,\bu\rangle_1 + \gamma \langle g,\theta \rangle_2,
\]
where $\langle \cdot,\cdot \rangle_i$ denotes the duality product in $V_i$, $i=1,2$. With this representation, the usual norm for an element $h=(\f,g)$ in the dual space $V'=V_1'\times V_2'$ can also be written in the form
\begin{equation}
  \label{norminvprime}
  \|h\|_{V'} = \sqrt{\|\f\|_{V_1'}^2 + \gamma\|g\|_{V_2'}^2},
\end{equation}
where $\|\cdot\|_{V_i'}$ denotes the usual norm of the dual space $V_i'$, $i=1,2$.

Similarly, $H$ is identified with its dual $H'=H_1'\times H_2'=H_1\times H_2$ with a norm analogous to \eqref{norminvprime}, and we have the continuous and dense injections $V\subset H=H'\subset V$.

We rewrite the system \eqref{eq33}-\eqref{eq35} in the following functional form
\be\label{eq37}\frac{\rd\bu}{\rd t}+\nu A_1\bu+B_1(\bu,\bu)=g\alpha \theta\bfe_3+g\alpha (T_{b,\varepsilon}+T_0-T_1)\bfe_3,\quad\text{in $V_1'$},\ee
\be\label{eq38}\frac{\rd\theta}{\rd t}+\kappa A_2\theta+B_2(\bu,\theta)=-(\bu\cdot\nabla)T_{b,\varepsilon}+\kappa\Delta T_{b,\varepsilon},\quad\text{in $V_2'$},\ee
where $B_1(\cdot,\cdot):V_1\times V_1\rightarrow V_1'$ and $B_2(\cdot,\cdot):V_1\times V_2\rightarrow V_2'$ are the bilinear operators defined by duality as
\[\langle B_1(\bu,\bv),\bw\rangle_1=((\bu\cdot\nabla)\bv,\bw)_1 \quad \forall \bu,\;\bv,\;\bw \in V_1,\]
\[ \langle B_2(\bv,\theta),\tilde\theta\rangle_2=((\bv\cdot\nabla)\theta,\tilde\theta)_2 \quad \forall \bv\in V_1,\;\forall \theta,\; \tilde\theta\in V_2.\]
Furthermore, $A_1:V_1\rightarrow V_1'$ and $A_2:V_2\rightarrow V_2'$ are the linear operators defined by duality according to
\[\langle A_1\bv,\tilde\bv \rangle_1 = (\!(\bv,\tilde\bv)\!)_1, \quad \forall \bv,\;\tilde\bv \in V_1,\]
\[\langle A_2\theta,\tilde\theta \rangle_2 = (\!(\theta,\tilde\theta)\!)_2, \quad \forall \theta,\;\tilde\theta \in V_2.\]
Both these operators can be seen as positive and self-adjoint closed operators with compact inverse when restricted to their domain $D(A_j)$ in $H_j$, $j=1,2$, given by
\[D(A_1)=\{\bv \in V_1; \; A_1\bv \in H_1\}, \]
\[D(A_2)=\{\theta \in V_2; \; A_2\theta \in H_2\}. \]
We let $\lambda_1$ and $\lambda_2$ denote the smallest eigenvalues of each of these operators and set $\lambda_0 = \min\{\lambda_1,\lambda_2\}$.

We can also write equations \eqref{eq37}-\eqref{eq38} in the compact form
\[\frac{\rd z}{\rd t}+Az+B(z,z)+Rz=0\quad\text{in }V',\]
where
\[z=(\bu,\theta),\]
\[Az=(\nu A_1\bu,\kappa A_2\theta),\]
\[Rz=(-g\alpha \theta\bfe_3-g\alpha (T_{b,\varepsilon}+T_0-T_1)\bfe_3,(\bu\cdot\nabla)T_{b,\varepsilon}-\kappa\Delta T_{b,\varepsilon})\]
and
\[B(z,\tilde{z})=(B_1(\bu,\tilde{\bu}),B_2(\bu,\tilde{\theta})),\quad \forall z=(\bu,\theta),\,\tilde{z}=(\tilde{\bu},\tilde{\theta}).\]

The following definition provides a notion of weak solution to the problem \eqref{eq33}-\eqref{eq35}. We denote by $H_w$ the space $H$ endowed with the weak topology.

\begin{defs}\label{def-sol-benard}
Let $I\subset\mathds{R}$ be an interval. We say that $z=(\bu,\theta)$ is a \textit{weak solution of system \eqref{eq33}-\eqref{eq35}} if
\begin{itemize}
\item [(i)] $z\in L^2_{\text{loc}}(I;V)\cap L^\infty_{\text{loc}}(I;H)\cap\Cloc(I,H_w)$;
\item [(ii)] $\partial_t z\in L^{4/3}_{\text{loc}}(I;V')$;
\item [(iii)] $z$ satisfies
\[\frac{\rd z}{\rd t}+Az+B(z,z)+Rz=0\quad\text{in }V'\]
in the sense of distributions on $I$.
\item [(iv)] For almost every $t'\in I$, $z=(\bu,\theta)$ satisfies the following energy inequalities
\be\label{benard-ener2}\frac{1}{2}|\bu(t)|_1^2+\nu\int_{t'}^t\|\bu(s)\|_1^2\rd s\leq \frac{1}{2}|\bu(t')|_1^2+\int_{t'}^t(g\alpha \theta(s)\bfe_3,\bu(s))_1\rd s\ee
\begin{eqnarray}\label{benard-ener1}\frac{1}{2}|\theta(t)|_2^2+\kappa\int_{t'}^t\|\theta(s)\|_2^2\rd s&\leq& \frac{1}{2}|\theta(t')|_2^2+\int_{t'}^t\langle(\bu\cdot\nabla)T_{b,\varepsilon},\theta\rangle_2\rd s\nonumber\\
&&+\kappa\int_{t'}^t\langle\Delta T_{b,\varepsilon},\theta\rangle_2\rd s,
\end{eqnarray}
for every $t\in I$ with $t>t'$. The set of times $t'$ for which \eqref{benard-ener2} and \eqref{benard-ener1} are valid can be characterized as the points of strong continuity from the right of $\bu$ and $\theta$, and they form a set of total measure in $I$.
\item [(v)]\label{Bv} If $I$ is closed and bounded on the left, with left end point $t_0$, then $z$ is strongly continuous at $t_0$ from the right, i.e. $z(t)\rightarrow z(t_0)$ in $H$ as $t\rightarrow t_0^+$.
\end{itemize}
\end{defs}

For any $R\geq 0$, let $B_H(R)$ be the closed ball with radius $R$ in $H$ and denote by $B_H(R)_w$ the closed ball endowed with the weak topology. Based on Definition \ref{def-sol-benard}, we consider the following trajectory spaces associated to the B\'enard problem:

\begin{align}
\mU_I & =\{z\in\Cloc(I,H_w)\,|\,z\text{ is a weak solution on }I\}, \label{mUI} \\
\mU_I(R) & =\{z\in\Cloc(I,B_H(R)_w)\,|\,z\text{ is a weak solution on }I\}, \\
\mU_I^\sharp & =\{z\in\Cloc(I,H_w)\,|\,z\text{ is a weak solution on }\mathring{I}\}, \\
\mU_I^\sharp(R) & =\{z\in\Cloc(I,B_H(R)_w)\,|\,z\text{ is a weak solution on }\mathring{I}\},
\end{align}
where $\mathring{I}$ denotes the interior of the interval $I$.

By choosing $\gamma$ sufficiently large and $\varepsilon$ sufficiently small, one obtains suitable estimates for the weak solutions in $\mU_I$.

\begin{prop}\label{benard-est}
Let $I\subset\mathds{R}$ be an interval and let $z\in\mU_I$. Suppose $\gamma$ satisfies
\be\label{eq47}\gamma>\frac{4(g\alpha)^2}{\nu\kappa\lambda_0^2}\ee
and $\varepsilon$ satisfies
\be\label{eq48}0 < \varepsilon^2<\frac{\nu}{\gamma(T_0-T_1)^2}\left(\frac{\kappa}{4}-\frac{(g\alpha)^2}{\gamma\nu\lambda_0^2}\right).\ee
Then, for every $t'\in I$ for which \eqref{benard-ener2} and \eqref{benard-ener1} are valid, and for every $t\in I$ with $t>t'$, the following estimates hold
\be\label{eq39}|z(t)|_H^2\leq |z(t')|_H^2\Exp^{-\eta\lambda_0(t-t')}+2\frac{\kappa\gamma}{\eta\lambda_0}\frac{L_1L_2}{\varepsilon}(T_1-T_0)^2(1-\Exp^{-\eta\lambda_0(t-t')}),\ee
\be\label{eq40}\int_{t'}^t\|z(s)\|_V^2\rd s\leq \frac{1}{\eta}|z(t')|_H^2+2\frac{\kappa\gamma}{\eta}\frac{L_1L_2}{\varepsilon}(T_1-T_0)^2(t-t'),\ee
\be\label{eq41}\left(\int_{t'}^t\|\partial_t z(s)\|_{V'}^{4/3}\rd s\right)^{3/4}\leq \frac{C}{(\nu\kappa)^{3/8}}|z(t')|_H^2 + C \frac{(\nu\kappa)^{1/8}}{\lambda_0^{3/2}}(t-t')+C\frac{(\nu\kappa)^{5/8}}{\lambda_0^{1/2}},\ee
where $\eta=\min\{\nu,\kappa\}$ and $C$ is a nondimensional constant which depends on the parameters $\nu$, $\kappa$, $\lambda_0$, $g$, $\alpha$, $T_0$, $T_1$, $\gamma$, $\varepsilon$, $L_1$ and $L_2$ through nondimensional combinations of them.
\end{prop}

The \textit{a priori} estimates \eqref{eq39}-\eqref{eq41} allow us to prove the existence of weak solutions of the initial-value problem associated to system \eqref{eq33}-\eqref{eq35} in the sense of Definition \ref{def-sol-benard}. The proof follows in a way similar to the classical result of existence of Leray-Hopf weak solutions of the Navier-Stokes equations \cite{Lady,bookcf1988,temam84}. The choice of the background flow and of the parameters $\gamma$ and $\varepsilon$ were based on the formulation given in \cite{CRT2004} for the B\'enard problem in two dimensions (see also \cite{Kapustyan} for the three-dimensional case in which the boundary conditions are, however, fully homogeneous, with the flow driven instead by a forcing distributed within the domain). We then obtain the following result.

\begin{thm}\label{existweaksolbenard}
Let $z_0\in H$. Then there exists at least one weak solution of problem \eqref{eq33}-\eqref{eq35} in the sense of Definition \ref{def-sol-benard} satisfying $z(t_0)=z_0$.
\end{thm}

Estimates \eqref{eq39}-\eqref{eq41} also imply the following result (see \cite{FRT2010,FRT} for the analogous result in the Navier-Stokes equations case).

\begin{prop}
\label{mUIRsharp}
The set $\mU_I^\sharp(R)$ is a compact and metrizable space and it is the closure of the set $\mU_I(R)$ with respect to the topology of $\Cloc(I,H_w)$.
\end{prop}

Let us now consider $I\subset\mathds{R}$ as an interval closed and bounded on the left with left end point $t_0$. Define
\[R_0=2\frac{\kappa\gamma}{\eta\lambda_0}\frac{L_1L_2}{\varepsilon}(T_1-T_0)^2.\]
If $R\geq R_0$ and $z(t_0)\in B_H(R)$, it follows from \eqref{eq39} with $t=t_0$ that
\[z(t)\in B_H(R)\,,\,\,\forall t\geq t_0.\]
Thus, $z\in\mU_I(R)$.

We now proceed to show that the abstract framework developed in the previous section is valid for the B\'enard problem. We consider $X$ as $H_w$ and the abstract set of trajectories $\mU$ as the set of weak solutions $\mU_I$ defined in \eqref{mUI}.

First, observe that hypothesis (\ref{H1}) of Definition \ref{hypothesisH} is readily verified by the set $\mU_I$ thanks to Theorem \ref{existweaksolbenard}.

The following results aim at proving the validity of the remaining hypotheses, \eqref{H2} and \eqref{H3}, for the set $\mU_I$. The next result proves \eqref{H2}.

\begin{lem}\label{lem1}
If $K\subset H_w$ is a compact subset (in the weak topology of $H$), then $\Pi_{t_0}^{-1}K\cap\mU_I$ is relatively compact in $\Cloc(I,H_w)$.
\end{lem}
\begin{proof}
Given a compact subset $K\subset H_w$, let $R\geq R_0$ be sufficiently large  such that $K\subset B_H(R)$. Then, if $z\in\KUI$, it follows in particular that $\Pi_{t_0}z\in K\subset B_H(R)$. Thus, from inequality \eqref{eq39}, $z\in\mU_I(R)$. We then have
\[\KUI\subset\mU_I(R).\]
Taking the closure of this set with respect to the topology of $\Cloc(I,H_w)$ and using Proposition \ref{mUIRsharp} we find that
\[\fKUI\subset \overline{\mU_I(R)} = \mU_I^\sharp(R).\]
Then, since $\mU_I^\sharp(R)$ is compact in $\Cloc(I,H_w)$, we conclude that $\fKUI$ is also compact in this space.
\end{proof}

Now let $V: I\times \Cloc(I,H_w)\rightarrow \mathbb R$ be the function defined by
\be\label{eq46}V(t,z)=\begin{cases}\displaystyle\frac{1}{t-t_0}\int_{t_0}^{t}|z(s)|_H^2ds, & \text{for } t\neq t_0\medskip\\ |z(t_0)|_H^2, & \text{for } t=
t_0.\end{cases}\ee

We shall prove that this specific function $V$ satisfies all the conditions required by hypothesis \eqref{H3}. In fact, $V$ is even more regular, as property \eqref{i} of Lemma \ref{lem2} below shows, since it yields that $V$ is a lower semi-continuous function on $I\times\Cloc(I,H_w)$, implying immediately conditions \eqref{H3i}, \eqref{H3ii} and \eqref{H3iv} of \eqref{H3}. Properties \eqref{ii} and \eqref{iii} of Lemma \ref{lem2} prove conditions \eqref{H3iii} e \eqref{H3v} of Hypothesis \eqref{H3}.

\begin{lem}\label{lem2}The function $V$, defined in \eqref{eq46}, satisfies the following properties
\renewcommand{\theenumi}{\roman{enumi}}
\begin{enumerate}
\item \label{i} $V$ is lower semi-continuous on $I\times\Cloc(I,H_w)$;
\item \label{ii} $V(t_0,z)=V(t_0,\tilde{z})$, for every $z,\tilde{z}\in\Cloc(I,H_w)$ such that $z(t_0)=\tilde{z}(t_0)$;
\item \label{iii} For every compact $K\subset X$, $V(t_0,\cdot)$ is bounded on $\Pi_{t_0}^{-1}K\cap\mU_{I}$;
\end{enumerate}
\end{lem}
\begin{proof}

For every $k\in\mathds{N}$, let $P_k^i:H_i\rightarrow H_i$ be the Galerkin projector onto the first $k$ modes of the operator $A_i$. Then let $P_k:H\rightarrow H$ be defined by
\[P_k z=(P_k^1 \bu,P_k^2\theta)\,,\,\,\forall z=(\bu,\theta)\in H.\]
Now for every $k\in\mathds{N}$, consider the function $V_k:I\times\Cloc(I,H_w)\rightarrow\mathds{R}$ defined by
\[V_k(t,z)=\begin{cases}\displaystyle\frac{1}{t-t_0}\int_{t_0}^{t}|P_k z(s)|_H^2\rd s,
 & \text{for } t\neq t_0,\medskip\\
|P_k z(t_0)|_H^2, & \text{for } t=
t_0.\end{cases}\]
Using that  $z\mapsto|P_k z(t)|_H^2$ is a continuous function on $\Cloc(I,H_w)$, for every $t\in I$, and that $t\mapsto|P_k z(t)|_H^2$ is a continuous function on $I$, for every $z\in\Cloc(I,H_w)$, one can show that $V_k$ is a continuous function on $I\times\Cloc(I,H_w)$. Then, since
\[V_k(s,\tilde{z}) \nearrow V(s,\tilde{z})\,,\,\,\forall(s,\tilde{z})\in I\times\Cloc(I,H_w),\]
and each $V_k$ is in particular a lower semi-continuous function on $I\times\Cloc(I,H_w)$, then $V$ is also a lower semi-continuous function on this space (see \cite[Lemma 2.41]{AB}).

Item \eqref{ii} follows directly from the definition of $V$.

Finally, for the proof of item \eqref{iii}, consider a compact set $K\subset H_w$ and let $R\geq 0$ be such that $K\subset B_H(R)$. Then,
\[V(t_0,z)=|z(t_0)|_H^2\leq R\,,\,\,\forall z\in \Pi^{-1}_{t_0}K\cap \mU_I.\]
\end{proof}

The next lemma proves condition \eqref{H3vi} of Hypothesis \eqref{H3}.
\begin{lem}\label{lem3}
For all compact $K\subset H_w$, the trajectory space $\mU_I$ has the following
characterizations,
\begin{eqnarray*}
\Pi_{t_0}^{-1}K\cap\mU_I&=&\left\{z\in\overline{\Pi_{t_0}^{-1}K\cap\mU_I}\,|\,\liminf_{t\rightarrow t_0^+}V(t,z)\leq V(t_0,z)\right\}\\
&=&\left\{z\in\overline{\Pi_{t_0}^{-1}K\cap\mU_I}\,|\,\limsup_{t\rightarrow t_0^+}V(t,z)\leq V(t_0,z)\right\}.
\end{eqnarray*}
\end{lem}
\begin{proof}
We first prove the characterization with the $\liminf$. Then the second characterization, with the $\limsup$, is readily verified.

Let $K\subset H_w$ be a compact set and let $z\in \Pi^{-1}_{t_0}K\cap \mU_I$.
It is clear that $z\in \overline{\Pi^{-1}_{t_0}K\cap \mU_I}$. Moreover, since
$z\in \mU_I$, from Definition \ref{def-sol-benard} it follows that $z$ is strongly continuous at $t_0$ from the right, which implies that $V(t,z)$ converges to $V(t_0,z)$ as $t\rightarrow t_0^+$. Then, in particular,
\[\liminf_{t\rightarrow t_0^+}V(t,z)\leq V(t_0,z).\]

On the other hand, consider $z\in\overline{\Pi^{-1}_{t_0}K\cap \mU_I}$ satisfying
\be\label{eq44}\liminf_{t\rightarrow t_0^+}V(t,z)\leq V(t_0,z).\ee
Since $\Pi_{t_0}$ is continuous and $K$ is closed we have that
\[\fKUI\subset \overline{\Pi_{t_0}^{-1}K}=\Pi_{t_0}^{-1}K,\]
so that $z\in\Pi_{t_0}^{-1}K$.

Let $R\geq 0$ be sufficiently large such that $K\subset B_H(R)$. Then, as in the proof of Lemma \ref{lem1},
\[\fKUI\subset\mU_I^{\sharp}(R),\]
which implies that $z$ is a weak solution on $\mathring{I}$.
Therefore, in order to show the first characterization, it only remains to prove that $z$ is strongly continuous in $H$ at $t_0$ from the right. Here we follow the proof of \cite[Lemma 2.3]{FRT} for a similar result for the Navier-Stokes equations. First observe that, since $z=(\bu,\theta)\in\Cloc(I,H_w)$, then in particular $\bu\in\Cloc(I,(H_1)_w)$ and $\theta\in\Cloc(I,(H_2)_w)$. Thus,
\[|\bu(t_0)|_1\leq\liminf_{t\rightarrow t_0^+}|\bu(t)|_1\]
and
\[|\theta(t_0)|_2\leq\liminf_{t\rightarrow t_0^+}|\theta(t)|_2,\]
which imply that
\be\label{eq42}\liminf_{t\rightarrow t_0^+}\left(\frac{1}{t-t_0}\int_{t_0}^t|\bu(s)|_1^2\rd s\right)-|\bu(t_0)|_1^2\geq 0\ee
and
\be\label{eq43}\liminf_{t\rightarrow t_0^+}\left(\frac{1}{t-t_0}\int_{t_0}^t|\theta(s)|_2^2\rd s\right)-|\theta(t_0)|_2^2\geq 0.\ee

From \eqref{eq44}, we obtain that
\begin{multline}\label{eq49}
\liminf_{t\rightarrow t_0^+}\left(\frac{1}{t-t_0}\int_{t_0}^t|\bu(s)|_1^2\rd s\right)-|\bu(t_0)|_1^2+\\
+\gamma\liminf_{t\rightarrow t_0^+}\left(\frac{1}{t-t_0}\int_{t_0}^t|\theta(s)|_2^2\rd s\right)-\gamma|\theta(t_0)|_2^2\\
\leq\liminf_{t\rightarrow t_0^+}V(t,z)-V(t_0,z)\leq 0.
\end{multline}

Thus, from \eqref{eq42}, \eqref{eq43} and \eqref{eq49} it follows that
\be\label{eq52}\liminf_{t\rightarrow t_0^+}\frac{1}{t-t_0}\int_{t_0}^t|\bu(s)|_1^2\rd s\leq |\bu(t_0)|_1^2,\ee
and
\be\label{eq53}\liminf_{t\rightarrow t_0^+}\frac{1}{t-t_0}\int_{t_0}^t|\theta(s)|_2^2\rd s\leq |\theta(t_0)|_2^2.\ee
From \eqref{eq52}, there is a sequence $\{t_n\}_n\subset I$ such that
$t_n\rightarrow t_0^+$ and
\[\frac{1}{t_n-t_0}\int_{t_0}^{t_n}(|\bu(s)|_1^2
-|\bu(t_0)|_1^2)\rd s\leq \frac{1}{n}.\]
This implies that for each $n\in \mathbb N$ there exists a subset $I_n\subset (t_0,t_n)$ of positive Lebesgue measure such that
\begin{equation}\label{conv}
|\bu(t'_n)|_1^2-|\bu(t_0)|_1^2\leq \frac{1}{n},
\end{equation}
for all $t'_n\in I_n$.
Thus,
\be\label{eq45}\limsup_{t_n'\rightarrow t_0}|\bu(t_n')|_1^2\leq|\bu(t_0)|_1^2.\ee

Since $z\in\mU_I^\sharp$ and $I_n$ is contained in $\mathring{I}$ and has positive Lebesgue measure, for all $n\in \mathbb N$ we can find $t'_n\in I_n$ such that $z$ satisfies inequality \eqref{benard-ener2} starting at the time $t_n'$, i.e.
\begin{eqnarray*}\frac{1}{2}|\bu(t)|_1^2+\nu\int_{t_n'}^t\|\bu(s)\|_1^2\rd s &\leq& \frac{1}{2}|\bu(t_n')|_1^2+\int_{t_n'}^t(g\alpha \theta(s)\bfe_3,\bu(s))_1\rd s,\nonumber\\
&&+\int_{t_n'}^t(g\alpha  T_{b,\varepsilon}\bfe_3,\bu(s))_1\rd s,
\end{eqnarray*}
for all $t\in I$ with $t>t'_n$.
Taking the $\limsup$ as $t_n'\rightarrow t_0^+$ and using \eqref{eq45} and the fact that $z\in L^2_{\text{loc}}(I, V)\cap L^\infty_{\text{loc}}(I,H)$, we then obtain that
\begin{eqnarray*}
\frac{1}{2}|\bu(t)|_1^2+\nu\int_{t_0}^t\|\bu(s)\|_1^2\rd s&\leq& \frac{1}{2}|\bu(t_0)|_1^2+\int_{t_0}^t(g\alpha \theta(s)\bfe_3,\bu(s))_1\rd s\nonumber\\
&&+\int_{t_0}^t(g\alpha T_{b,\varepsilon}\bfe_3,\bu(s))_1\rd s,
\end{eqnarray*}
for all $t\in I$, with $t>t_0$, which implies that $\bu$ is strongly continuous in $H_1$ at $t_0$ from the right.

By using \eqref{eq53} we can show analogously that $\theta$ satisfies inequality \eqref{benard-ener1} with $t'=t_0$, which implies that $\theta$ is strongly continuous in $H_2$ at $t_0$ from the right. Thus, $z$ is strongly continuous in $H$ at $t_0$ from the right, and this finishes the proof of the characterization of $\KUI$ with the $\liminf$.

Now suppose that $z\in\fKUI$ and $\limsup_{t\rightarrow t_0^+} V(t,z)\leq V(t_0,z)$. Thus,
\[\liminf_{t\rightarrow t_0^+} V(t,z)\leq\limsup_{t\rightarrow t_0^+} V(t,z)\leq V(t_0,z).\]
Then, by the first characterization it follows that $z\in\KUI$.

On the other hand, if $z\in\KUI$ then again $z$ is strongly continuous at $t_0$ from the right, which implies in particular that
\[\limsup_{t\rightarrow t_0^+} V(t,z)\leq V(t_0,z).\]
This proves the characterization with the $\limsup$.
\end{proof}

Finally, the following result proves the remaining condition \eqref{H3vii} of Hypothesis \eqref{H3}. In fact, it proves a stronger condition, with the $\limsup$ in time instead of the $\liminf$, and with the supremum in $z$ over $\mU_I$ instead of just over $\Pi_{t_0}^{-1}K\cap\mU_I$.
\begin{lem}\label{lem4}
Let $V:I\times\Cloc(I,H_w)$ be the function defined in \eqref{eq46} and let $I\subset \mathbb{R}$ be an interval closed and bounded on the left with left end point $t_0$. Then $\mU_I$ satisfies the following property:
\[\limsup_{t\rightarrow t_0^+}\sup_{z\in\mU_I}(V(t,z)-V(t_0,z))\leq 0.\]
\end{lem}
\begin{proof}
Let $z\in \mU_I$. Since $I$ is closed and bounded on the left with left end point $t_0$, it follows from Definition \ref{def-sol-benard} that $z$ is strongly continuous at $t_0$ from the right. This implies, in particular, that inequalities \eqref{benard-ener2} and \eqref{benard-ener1} are valid for $t'=t_0$. Thus, using these inequalities and conditions \eqref{eq47} and \eqref{eq48} on $\gamma$ and $\varepsilon$, it is not difficult to obtain that
\[|z(t)|_H^2+\eta\int_{t_0}^t\|z(s)\|_V^2\rd s\leq |z(t_0)|_H^2+2\kappa\gamma\frac{L_1L_2}{\varepsilon}(T_1-T_0)^2(t-t_0),\]
where $\eta=\min\{\nu,\kappa\}$.

Then, discarding the nonnegative integral term on the left-hand side of this inequality and integrating with respect to the time variable on $[t_0,t]$, we have
\[\int_{t_0}^t(|z(s)|_H^2-|z(t_0)|_H^2)\rd s\leq 2\kappa\gamma\frac{L_1L_2}{\varepsilon}(T_1-T_0)^2\int_{t_0}^t(s-t_0)\rd s.\]
Since the inequality above is valid for all $z\in\mU_I$, using the definition of $V$ it follows that
\[\sup_{z\in\mU_I}(V(t,z)-V(t_0,z))\leq 2\kappa\gamma\frac{L_1L_2}{\varepsilon}(T_1-T_0)^2\frac{1}{t-t_0}\int_{t_0}^t(s-t_0)\rd s.\]
Taking the $\limsup$ as $t\rightarrow t_0^+$, we finally obtain that
\[\limsup_{t\rightarrow t_0^+} \sup_{z\in\mU_I}(V(t,z)-V(t_0,z))\leq 0.\]
\end{proof}

According to the abstract framework developed in Section 3, the results above provide us with the sufficient conditions to prove the existence of a $\mU_I$-trajectory statistical solution of the B\'enard problem \eqref{eq33}-\eqref{eq35} for a given initial data.

\begin{thm}
\label{existenceBenard}
Let $I\subset\mathds{R}$ be an interval closed and bounded on the left with left end point $t_0$ and let $\mU_I$ be the set of weak solutions of problem \eqref{eq33}-\eqref{eq35} on $I$. If $\mu_0$ is a tight Borel probability measure on $H_w$ then there exists a $\mU_I$-trajectory statistical solution $\rho$ on $\Cloc(I,H_w)$ such that $\Pi_{t_0}\rho=\mu_0$.
\end{thm}
\begin{proof}
As mentioned along this section, we consider $X$ as the Hausdorff topological space $H_w$ and $\mU$ as the set of weak solutions $\mU_I$ defined in \eqref{mUI}. Theorem \ref{existweaksolbenard} proves that $\mU_I$ satisfies \eqref{H1} of Definition \ref{hypothesisH}. Hypothesis \eqref{H2} is proved in Lemma \ref{lem1}. From Lemma \ref{lem2}, $\mU_I$ satisfies items \eqref{H3i} to \eqref{H3v} of \eqref{H3}. Item \eqref{H3vi} is proved in Lemma \ref{lem3} and item \eqref{H3vii} is proved in Lemma \ref{lem4}.

Therefore, $\mU_I$ satisfies hypothesis (H) of Definition \ref{hypothesisH} and, from Theorem \ref{existence}, there exists a $\mU_I$-trajectory statistical solution $\rho$ on $\Cloc(I,H_w)$ satisfying $\Pi_{t_0}\rho=\mu_0$.
\end{proof}

Once a trajectory statistical solution $\rho$ is obtained for the B\'enard problem, as given by Theorem \ref{existenceBenard}, one can show, in a way similar to the proof for the Navier-Stokes equations \cite{FRT2010, FRT}, that the family of projections $\{\Pi_t\rho\}_{t\in I}$ is a statistical solution in the phase space, in the corresponding sense of the B\'enard problem. Our main interest, however, is in developing also an abstract formulation for statistical solutions in phase space, as described after the proof of Theorem \ref{existence}, and that will, of course, include the B\'enard problem. This will be presented elsewhere.

\section*{Acknowledgments}
The authors would like to thank Professors Dinam\'erico Pombo, for enlightening discussions about general topology, and Fabio Ramos, for bringing to our attention the works of Topsoe. The last author, R. Rosa, is also greatly indebted to Professors Roger Temam and Ciprian Foias, for their continued mentoring and support and, in particular, for all he has learned from them on the subject.

\end{document}